\documentclass[12pt]{colt2019} 

\usepackage{lipsum}

\makeatletter
\renewcommand*{\@fnsymbol}[1]{\ensuremath{\ifcase#1 \or
		\or \ddagger\ddagger \else\@ctrerr\fi}}
\makeatother

\usepackage{bm}

\usepackage{wrapfig}

\usepackage{graphicx}
\graphicspath{{figures/}}
\usepackage{epstopdf}

\usepackage[font={scriptsize}]{caption}
\usepackage[font={scriptsize}]{subcaption}

\newlength{\noteWidth}
\long\def\notes#1{\ifinner
           {\footnotesize #1}
           \else
           \marginpar{\parbox[t]{\noteWidth}{\raggedright\footnotesize #1}}
       \fi\typeout{#1}}

 \def\notes#1{\typeout{read notes: #1}}  

\def\urls#1{{\small \url{#1}}}

\def\spm#1{\notes{SPM:  #1}}
\def\ad#1{\notes{AD:  #1}}

\def\bl#1{{\color{blue}#1}}
\def\rd#1{{\color{red}#1}}

\def\Ebox#1#2{%
\begin{center}
\includegraphics[width= #1\hsize]{#2} 
\end{center}}

\newcounter{rmnum}
\newenvironment{romannum}{\begin{list}{{\upshape (\roman{rmnum})}}{\usecounter{rmnum}
\setlength{\leftmargin}{16pt}
\setlength{\rightmargin}{2pt}
\setlength{\itemindent}{10pt}
}}{\end{list}}

\def\Noise{\Delta}

\def\lambdamax{\lambda_{\text{\rm\scriptsize max}}}

\def\pie{\varpi}
\def\barf{\bar{f}}

\def\U{{\sf U}}
\def\state{{\sf X}}

\def\tilA{\tilde A}
\def\tilb{\tilde b}

\def\tiltheta{\tilde\theta}

\def\tilclL{\tilde \clL}

\def\trace{\text{\rm trace}\,}

\def\ddt{{\mathchoice{\FRAC{1}{d}{dt}}%
{\FRAC{1}{d}{dt}}%
{\FRAC{3}{d}{dt}}%
{\FRAC{3}{d}{dt}}}}

\def\eqdef{\mathbin{:=}}

\def\sq{$\Box$}

\def\qed{\ifmmode\Box\else{\unskip\nobreak\hfil
\penalty50\hskip1em\null\nobreak\hfil\sq
\parfillskip=0pt\finalhyphendemerits=0\endgraf}\fi\par\medbreak}

\def\argmin{\mathop{\rm arg{\,}min}}

\def\bfDelta{\bfmath{\Delta}}

\def\bfalpha{\bfmath{\alpha}}

\def\bfmA{\bfmath{A}}

 \DeclareBoldMathCommand\bfmclM{{\cal M}}
 \DeclareBoldMathCommand\bfmell{\ell}
 
 \DeclareBoldMathCommand\bflambda{\lambda}

 \DeclareBoldMathCommand\bfzeta{\zeta}
 
 \DeclareBoldMathCommand\bfmp{p}
 \DeclareBoldMathCommand\bfmr{r}
 \DeclareBoldMathCommand\bfmy{y}

\def\nd{\ell}  
\def\ninp{\ell_u}

\def\bfmath#1{{\mathchoice{\mbox{\boldmath$#1$}}%
{\mbox{\boldmath$#1$}}%
{\mbox{\boldmath$\scriptstyle#1$}}%
{\mbox{\boldmath$\scriptscriptstyle#1$}}}}

\def\Lemma#1{Lemma~\ref{#1}}
\def\Prop#1{Proposition~\ref{#1}}

\def\Section#1{Section~\ref{#1}}

\def\Fig#1{Fig.~\ref{#1}}

\def\basis{\psi}

\def\clB{{\cal B}}

\def\clE{{\cal E}}
\def\clF{{\cal F}}

\def\clL{{\cal L}}
\def\clM{{\cal M}}

\def\clX{{\cal X}}

\def\hab{\widehat b}

\def\haA{\widehat A}

\def\haB{\widehat B}
\def\haD{\widehat D}

\def\haM{\widehat M}

\def\haW{\widehat W}

\def\haSigma{\widehat \Sigma}

\def\bfmxi{{\mbox{\protect\boldmath$\xi$}}}

\def\bfmX{{\mbox{\protect\boldmath$X$}}}

\def\half{{\mathchoice{\textstyle \frac{1}{2}}%
{\frac{1}{2}}%
{\hbox{\tiny $\frac{1}{2}$}}%
{\hbox{\tiny $\frac{1}{2}$}} }}

\def\ddt{\frac{d}{dt}}

\def\transpose{{\hbox{\tiny\it T}}}

\def\epsy{\varepsilon}

\def\varble{\,\cdot\,}

\def\Prob{{\sf P}}

\def\Expect{{\sf E}}

\newcommand{\field}[1]{\mathbb{#1}}

\renewcommand{\Re}{\mathbb{R}}

\def\ind{\field{I}}

\def\process{\clX}
\DeclareBoldMathCommand\bfprocess{{\cal X}}

\def\bartheta{\bar\theta}
\def\mtum{\mu}
\def\Mtum{M}

\def\barPhi{\overline{\Phi}}

\def\Ascale{\zeta}

\newsavebox{\junk}
\savebox{\junk}[1.6mm]{\hbox{$|\!|\!|$}}

\title{Optimal Matrix Momentum Stochastic Approximation 
\\
and Applications to Q-learning}

\coltauthor{\begin{center}
Adithya M. Devraj$^{1}$, \,\,  Ana Bu\v si\'c$^{2}$, \,\, and \,\, Sean P. Meyn$^{1}$
\end{center} 
\thanks{Funding from the National Science Foundation award EPCN 1609131,   and French National Research Agency grant ANR-16-CE05-0008 \newline
$^{1}$Department of Electrical and Computer Engineering, University of Florida, Gainesville, FL 32611
\newline$^{2}$Inria and the Computer Science Department of \'Ecole Normale Sup\'erieure, 75005 Paris, France
}
}

\begin{document}

\maketitle
	
\begin{abstract}

Acceleration is an increasingly common theme in the stochastic optimization literature.   The two most common examples are Nesterov's method, and Polyak's momentum (heavy ball) technique.  In this paper two new algorithms are introduced for root finding problems:   1) PolSA is a root finding algorithm with specially designed matrix momentum,  and 2) NeSA can be regarded as a variant of Nesterov's algorithm, or a simplification of the PolSA algorithm.  
 The PolSA algorithm is new even in the context of
optimization (when cast as a root finding problem).

The research surveyed in this paper is motivated by applications to reinforcement learning.  It is well known that most variants of TD- and Q-learning may be cast as SA (stochastic approximation) algorithms, and the tools from general SA theory can be used to investigate convergence and bounds on convergence rate.   In particular, the asymptotic variance is a common metric of performance for SA algorithms, and is also one among many metrics used in assessing the performance of stochastic optimization algorithms.

There are two well known stochastic approximation techniques that are known to have optimal asymptotic variance:  the Ruppert-Polyak averaging technique, and stochastic Newton-Raphson (SNR).

The former algorithm can have extremely bad transient performance, and the latter can be computationally  expensive.   It is demonstrated here that parameter estimates from the new PolSA algorithm \emph{couple} with those of the ideal (but more complex) SNR algorithm.   The new algorithm is thus a third approach to obtain optimal asymptotic covariance.   

These strong results require assumptions on the model.  A linearized model is considered, and the noise is assumed to be a martingale difference sequence.
Numerical results are obtained in  a non-linear setting that is the motivation for this work:   In PolSA implementations of  Q-learning [a nonlinear algorithm]  it is  observed 
that coupling occurs  with SNR in this non-ideal setting.  The performance of NeSA is also very good compared to recent and standard variants of Q-learning.

\end{abstract}

\newpage

%
%
%

%
%
%
%
%
%
%
%
%
%
%
%
%
%
%
%

\section{Introduction}
\label{s:intro}


The general goal of this paper is the efficient computation of the root of a vector valued function:
obtain the solution $\theta^*\in\Re^d$ to the $d$-dimensional equation:
\begin{equation}
\barf(\theta^*) =0 .
\label{e:goal}
\end{equation}
It is assumed that the function  $\barf\colon\Re^d\to\Re^d$ is expressed as an expectation:
$\barf(\theta) = \Expect[f(\theta,\process)]$,  where  $f\colon\Re^d\times\Re^m \to \Re^d$ and $\process $ is an $ \Re^m$-valued  random variable. \textit{The function $\barf$ is not necessarily equal to a gradient, so the setting of this paper goes beyond optimization.}

The stochastic approximation (SA) literature contains a large collection of tools to construct algorithms,  and obtain bounds on their convergence rate.   
In this paper we show how algorithms with optimal rate of convergence can be constructed based on a synthesis of techniques from classical SA theory combined with variants of momentum algorithms pioneered by Polyak  (\cite{pol64,pol87}).

The algorithms and analysis in this paper admit application to both stochastic optimization and reinforcement learning.     
As in much of this literature, it is assumed in this paper that  there is a sequence of random functions $\{f_n\}$ satisfying  for each $\theta\in\Re^d$, 
\begin{equation}
\barf(\theta)
=
\lim_{n\to\infty} \frac{1}{n}\sum_{k=1}^{n} f_k(\theta)
=
\lim_{n\to\infty} \Expect[f_n(\theta)]  \,,
\label{e:barf}
\end{equation}
where the first limit is in the a.s.\ sense.  

Three general classes of algorithms are investigated in this work.   Each is defined with respect to a non-negative scalar gain sequence $\{\alpha_n\}$,  and two include $d\times d$ matrix sequences $\{G_n\}, \{\Mtum_n\}$.
  For each algorithm,    the difference sequence is denoted $
\Delta\theta_n \eqdef \theta_n-\theta_{n-1}$,  $ n\ge 0$,
with given initial condition $\theta_0 =\theta_{-1}$.

\vspace{0.1in}
\paragraph{1.  Stochastic approximation with matrix gain} 
\begin{equation} 
\Delta\theta_{n+1} =  \alpha_{n+1} G_{n+1} f_{n+1}(\theta_n)\,
\label{e:SAa}
\end{equation} 

\paragraph{2.  Matrix Heavy-Ball Stochastic approximation} 
\begin{equation}
	\Delta\theta_{n+1} =   \Mtum_{n+1} \Delta\theta_n  
		+   \alpha_{n+1}  G_{n+1} f_{n+1}(\theta_n)
\label{e:PolSAa}
\end{equation}

\paragraph{3.  Nesterov Stochastic approximation (NeSA)}
For a fixed scalar $\Ascale>0$,
\begin{equation}
\begin{aligned}
\Delta\theta_{n+1} =  \Delta\theta_n 
& + \Ascale [ f_{n+1}(\theta_n)-f_{n+1}(\theta_{n-1}) ]
 + \Ascale \alpha_{n+1} f_{n+1}(\theta_n)
\label{e:NeSA}
\end{aligned}
\end{equation}

If $G_n\equiv I$,  then  \eqref{e:SAa} is the classical algorithm of  
\cite{robmon51a}.  In  Stochastic Newton Raphson (SNR)  and the more recent \textit{Zap SNR} (\cite{rup85,devmey17b,devmey17a}),  the matrix sequence $\{G_n\}$ is chosen to be an approximation of $-[\partial \barf(\theta_n)]^{-1}$.    Stability of the algorithm has been demonstrated in application to Q-learning  (\cite{devmey17b,devmey17a});  a non-trivial result, given that Q-learning is cast as root finding and not an optimization problem. 

%

The  matrix heavy ball algorithm \eqref{e:PolSAa} coincides with the heavy-ball method when $\{M_n\}$ is a sequence of scalars  (\cite{pol64,pol87,loiric17}).     Justification for the special form \eqref{e:NeSA} in NeSA is provided in the next section.

As in many previous papers in the context of high-dimensional optimization (\cite{loiric17}) and SA  (\cite{kontsi04,kusyin97,bor08a}), parameter error analysis is restricted to a linear setting:
\begin{equation*}
 f_{n+1} (\theta_n) = A_{n+1} \theta_n -b_{n+1} =  A(\theta_{n} - \theta^*) + \Noise_{n+1}
\end{equation*}
in which $(A_n,b_n)$ is a stochastic process with common mean $(A,b)$, and  for $n \geq 1$     
\begin{equation*}
\Noise_{n+1} \eqdef          \tilA_{n+1} (\theta_n - \theta^*)  +\Noise_{n+1}^* \quad \textit{with}\quad \Noise_{n+1}^* \eqdef  f_{n+1}(\theta^*) = A_{n+1} \theta^* - b_{n+1}
\end{equation*}
and the tilde always denotes deviation: 
$
\tilA_{n+1} \eqdef A_{n+1} - A$.

\paragraph{Goals} 
The main goal is to design algorithms with  (i)  \textit{fast convergence} to zero of the error sequence: $ \tiltheta_n \eqdef \theta_n - \theta^*$,  and (ii) \textit{low computational complexity}.


Rates of convergence are well understood for the SA recursion \eqref{e:SAa}. It is known that the 
Central Limit Theorem and Law of the Iterated Logarithm hold under general conditions, and the \textit{asymptotic covariance} appearing in these results can be expressed as the limit
\begin{equation}
\Sigma^\theta=\lim_{n\to\infty} \Sigma^\theta_n \eqdef \lim_{n\to\infty} n \Expect[\tiltheta_n\tiltheta_n^\transpose]\, .
\label{e:SAsigma}
\end{equation}   
The LIL    may be most interesting in terms of bounds (\cite{kusyin97,kovsch03});  it may not be as satisfying as a   Hoeffding or PAC-style finite-$n$ bound, but   there are no such bounds for Markovian models with useful constants (see e.g. \cite{glyorm02});
  reinforcement learning is typically cast in a Markov setting.


A \textit{necessary} condition for quick convergence is that the CLT or LIL hold with small asymptotic covariance.     Again, for the SA recursion \eqref{e:SAa},  optimization of this parameter is well-understood.  Denote by  $ \Sigma^G$ the asymptotic covariance for  \eqref{e:SAa} 
with $G_n \equiv  G$.   When this is finite, it admits a representation in terms of  the asymptotic covariance of the noise:
\begin{equation}
\Sigma^\Delta 
= \lim_{n\to\infty} \frac{1}{n} \Expect\Bigl[ \,\, \Bigl(\sum_{k=1}^n  \Noise_k^*\Bigr)   \Bigl(\sum_{k=1}^n  \Noise_k^*\Bigr) ^\transpose  \,\, \Bigr] 
\label{e:SigmaDelta}
\end{equation} 
In particular, the choice $G = G^* \eqdef -A^{-1}$ is a special case of SNR, for which asymptotic covariance admits the explicit form
\begin{equation}
\Sigma^* \eqdef A^{-1} \Sigma^\Delta {A^{-1}}^{\transpose}  
\label{e:SigmaStar}
\end{equation} 
This is optimal:   the difference $ \Sigma^G - \Sigma^*$ is positive semi-definite for any $G$ (\cite{benmetpri90,kusyin97,bor08a}).  

\textit{What about computational complexity?}
In realistic applications of SNR, the gain sequence will be of the form $G_n=-\haA_n^{-1}$, where $\{\haA_n\}$ are approximations 
(Monte-Carlo estimates) of  the mean $A$.  In a nonlinear model,   $\haA_n$
is an approximation of $\partial \barf(\theta_n)$, obtained using the two time-scale algorithm of \cite{devmey17b,devmey17a}.  
  The resulting complexity is a barrier to application in high dimension.  Steps towards resolving this  obstacle are presented in this paper:
\begin{romannum}
\item 
The parameters in the matrix heavy ball SA algorithm can be designed so that the error sequence enjoys all the attractive properties of SNR, but without the need for   matrix inversion.

\item  NeSA   is often simpler than the matrix heavy ball method in applications to RL.
A formula for the asymptotic covariance of a variant of NeSA is obtained in this paper.   
While not equal to $\Sigma^*$,  the reduced complexity makes it a valuable option. 
\end{romannum}
These conclusions are established in Propositions~\ref{t:couple},  \ref{t:OptVarZapHBAn}
and \ref{t:OptVarZapHB}  for linear recursions,
and illustrated in numerical examples for new Q-learning algorithms that are introduced in \Section{s:num}.
The assumptions of the main results are violated in application to $Q$-learning since the particular root finding problem is non-linear.   Nevertheless, coupling is seen between PolSA and Zap Q-learning in all of the numerical experiments conducted so far.

 Nesterov's acceleration and the heavy-ball method are both second order algorithms, but their relationship has been unclear until now.   In this paper we propose a new understanding of the relationship, which is only possible through the introduction of \emph{matrix momentum}.   
We show that the matrix momentum algorithm PolSA can be interpreted as a linearization of a particular formulation of Nesterov's method.  We further show that the PolSA algorithm approximates (stochastic) Newton Raphson, thus establishing connections between the three algorithms: Nesterov's accleration, PolSA, and Newton Raphson. This may not only help explain the success of Nesterov's acceleration, but may also lead to new algorithms in other application domains such as empirical risk minimization (ERM).\footnote{The key theoretical results in this paper \textit{are not} directly applicable to these problems -- an explanation is given in \Section{s:app}. }

\paragraph{Literature survey}   

The present paper is built on a vast literature on optimization  (\cite{nes83,pol64,pol87,nes12})  and stochastic approximation
(\cite{kontsi04,kusyin97,bor08a,rup85,pol90,poljud92}).  The work of Polyak is central to both thrusts:  the introduction of momentum, and techniques to minimize variance in SA algorithms. The reader is referred to (\cite{devmey17a})  for a survey on SNR and the more recent Zap SNR algorithms, which are also designed to achieve minimum asymptotic variance.
 
In the stochastic optimization literature, the goal is to minimize an expectation of a function. In connection to \eqref{e:barf}, each $f_n$ can be viewed as an unbiased estimator of the gradient of the objective. The papers  (\cite{moubac11,bacmou13,gadpansaa18,duc16,jaikakkidnetsid17}) establish the optimal  convergence rate of $O(1/\sqrt{n})$ for various stochastic optimization algorithms.   
 



In ERM (empirical risk minimization) literature, the sample path limit in \eqref{e:barf} is replaced by a finite average
(\cite{Katyusha16,SAGA14,jaikakkidnetsid17}):
$
\barf_n(\theta) =  n^{-1}\sum_{k=0}^{n-1} f_k(\theta)$. Denoting 
$\theta_n^* = \argmin_\theta \barf_n(\theta) $,  
under general conditions it can be shown that the sequence of ERM optimizers  $\{\theta_n^*\}$ is convergent to $\theta^*$, and has optimal asymptotic covariance  (a survey and further discussion   is presented in  \cite{jaikakkidnetsid17}).

The   recent paper  \cite{jaikakkidnetsid17} is most closely related to the present work, considering the shared goal of optimizing the asymptotic covariance, along with rapidly vanishing transients through algorithm design. The paper restricts to stochastic optimization rather than the general root finding problems considered here, thus ruling out application to many reinforcement learning algorithms such as TD- and Q- learning (\cite{tsi94a,tsiroy97a,kontsi04}).
The metric for performance is slightly different,   focusing on the rate of convergence of the \textit{expected loss},  for which they obtain bounds for each iteration $n$ of the algorithm.   Along with establishing that the algorithm achieves optimal asymptotic variance, they also obtain tight bounds on the regret. The algorithm uses an averaging technique similar to the one in (\cite{poljud92}), that helps achieve the optimal variance.

The algorithms presented in this work achieve the optimal asymptotic covariance, 
are not restricted to  optimization, 
 and we believe that in many applications they will be simpler to implement.  This is especially true for the NeSA algorithm  applied to Q-learning.

\section{Momentum methods and applications}  
\label{s:momentum}

\subsection{Motivation \& Insights}   


Consider first the deterministic root-finding problem.  This will bring insight into the relationship between the three algorithms  (\ref{e:SAa}--\ref{e:NeSA})  discussed in the introduction. 
The notation $f: \Re^d \to \Re^d$ is used in place of $\barf$ in this deterministic setting.   
The goal remains the same:  find the vector $\theta^*\in\Re^d$ such that $f(\theta^*) =0$.

Deterministic variants of (\ref{e:SAa}--\ref{e:NeSA})   commonly considered in the literature are, respectively,  
\begin{align}
 &\text{\bf Successive approximation:} 
\qquad  
\Delta\theta_{n+1} =  \alpha f (\theta_n)
\label{e:GD}
\\
 &\text{\bf Polyak's heavy ball:} 
\qquad\qquad  
\Delta\theta_{n+1} =  \mtum\Delta\theta_n  + \alpha f (\theta_n)
\label{e:HB}
\\
 &\text{\bf Nesterov's acceleration:} 
\qquad  \Delta\theta_{n+1} =  \mtum  \Delta\theta_n  + \Ascale [ f(\theta_n) -  f(\theta_{n-1})]  + \alpha f (\theta_n)
\label{e:NA}
\end{align} 
where $\alpha,\mtum,\Ascale$ are positive constants.
Nesterov's algorithm was designed for extremal seeking, which is the special case $f=-\nabla J$ for a real-valued function $J$.   The recursion \eqref{e:NA} is the natural extension to the root-finding problem considered here.


The questions asked in this paper are posed in a stochastic setting, 
but analogous questions are: 
\begin{romannum} 
\item   Why restrict to a scalar momentum term $\mtum$, rather than a matrix $\Mtum$?


\item  Can online algorithms be designed to approximate the optimal momentum matrix?    If so, we   require tools to 
  investigate the performance of a given matrix sequence $\{\Mtum_n  \}$:
\begin{equation}
\Delta\theta_{n+1} =  \Mtum_{n+1}\Delta\theta_n + \alpha f (\theta_n)
\label{e:HBmatrix}
\end{equation} 
\end{romannum}
Potential answers are obtained by establishing   relationships between these deterministic recursions.  
{The heuristic relationships presented here are   justified for the stochastic models considered later in the paper.}

Consider the successive approximation algorithm \eqref{e:GD} under the assumption of global convergence:  $\theta_n\to\theta^* $  as $n\to\infty$.    Assume moreover that   $f \in C^1$ and Lipschitz, 
so that
\[
\begin{aligned}
\Delta\theta_{n+1} -\Delta\theta_n   & \approx  \alpha  \partial f\, (\theta_n) \Delta\theta_n
\\
& =   \alpha^2 \partial f\, (\theta_n)  f (\theta_{n-1})
\end{aligned}
\]
\spm{We would ned $d$ equations:  where each component of $\bartheta_n\in\Re^d$ lies on the line connecting the corresponding components of $\theta_{n+1}$ and $\theta_n$. 
}
 It follows that  
$
\| \Delta\theta_{n+1} -\Delta\theta_n\|  = O(\min\{\alpha^2,\alpha \|  \Delta\theta_n\| \})
$.
This suggests a \textit{heuristic}:   swap $\Delta\theta_{n+1} $ and $\Delta\theta_n$ in a given convergent algorithm to obtain a new algorithm that is simpler, but with desirable properties.  Applying this heuristic to   \eqref{e:HBmatrix} results in  
\[
\begin{aligned}
\Delta\theta_{n+1} 
& =  \Mtum_{n+1} {\Delta\theta_n} + \alpha f (\theta_n)
 \approx \Mtum_{n+1} {\Delta\theta_{n+1} } + \alpha f (\theta_n)
\end{aligned}
\]
Assuming that an inverse exists,  this becomes 
\[ 
 \Delta\theta_{n+1}   \approx     \alpha [I-\Mtum_{n+1}]^{-1}f (\theta_n)
\]
We thus arrive at a possible answer to the question of \emph{optimal momentum}:  
For the matrix sequence $ \Mtum_{n+1}    =  I + \alpha \partial f\, (\theta_n) $, the algorithm  \eqref{e:HBmatrix} can be expressed
\begin{equation}
\Delta\theta_{n+1} =  [I + \alpha \partial f\, (\theta_n) ]  \Delta\theta_n + \alpha f (\theta_n)
\label{e:OurLinearHBdet}
\end{equation}
The foregoing approximations suggest that this is an approximation of Newton-Raphson:
\[ 
 \Delta\theta_{n+1}   \approx  - [  \partial f \, (\theta_n)]^{-1}  f (\theta_n)
\]
Further approximations lead to different interpretations:
a
Taylor series argument shows that the recursion \eqref{e:OurLinearHBdet} is approximated by 
\begin{equation}
\Delta\theta_{n+1} =  \Delta\theta_n    + \alpha [f (\theta_n) -f (\theta_{n-1}) ]           
			     + \alpha f (\theta_n)
\label{e:HBNR}
\end{equation}
This is the special case of Nesterov's algorithm \eqref{e:NA} with $\mtum=1$ and $\Ascale = \alpha$.


Strong justification for the stochastic analog of \eqref{e:OurLinearHBdet} is provided through a coupling bound between the respective algorithms:  see \Prop{t:couple}.     It is found that similar transformations lead to new algorithms for reinforcement learning and other applications.  
\vspace{-0.1in}
\subsection{Optimal matrix momentum and PolSA}

Returning to the stochastic setting,  the PolSA algorithm considered in this paper is a special case of   matrix heavy ball SA \eqref{e:PolSAa}, and an analog of \eqref{e:OurLinearHBdet}:
\begin{equation}  
	\Delta\theta_{n+1} =   [I+    \Ascale \haA_{n+1}]  \Delta\theta_n  
		+   \alpha_{n+1}  \Ascale  f_{n+1}(\theta_n)
\label{e:PolSA} 
\end{equation} 
where $\Ascale>0$,  and $\{\haA_n\}$ are estimates of 
$
A(\theta)\eqdef \Expect[\partial f_n\, (\theta)]
$  (assumed independent of $n$).

The choice $G_n \equiv \Ascale I$ in \eqref{e:PolSAa} is imposed to simplify exposition;  in \Section{s:num} it is shown that a particular diagonal matrix gives much better performance in applications to Q-learning.

The main technical results are obtained for a linear model, so that 
\begin{align}
f_{n+1} (\theta_n)& = A_{n+1} \theta_n -b_{n+1} =  A(\theta_{n} - \theta^*) + \Noise_{n+1}
\label{e:fLin}
\\
\text{\it with} \quad
\Noise_{n+1}  & \eqdef          \tilA_{n+1} (\theta_n - \theta^*)  +\Noise_{n+1}^* \quad \textit{and}\quad \Noise_{n+1}^* \eqdef  f_{n+1}(\theta^*) = A_{n+1} \theta^* - b_{n+1}
\label{e:SigmaNoise}
\end{align}
In this case we denote $A=A(\theta)$.   Estimates are obtained  using
\begin{equation}
\haA_{n+1} = \haA_n + \frac{1}{n+1} (A_{n+1} - \haA_n)
\label{e:haAn}
\end{equation}
The SNR algorithm is \eqref{e:SAa} in which $G_n = \haA_n^{\dagger}$  (the Moore--Penrose pseudo inverse). 

Additional simplifying assumptions are imposed to ease analysis:
\begin{romannum}
\item[\textbf{(A1)}]  
The stochastic process  $(A_n,b_n)$ is wide-sense stationary, with common mean   $(A,b)$.

\item[\textbf{(A2)}]  $\{\tilA_n, \tilb_n\}$ are bounded martingale difference sequences, adapted to the filtration 
\begin{equation}
\clF_n \eqdef \sigma\{A_k,b_k : k\le n\}
\label{e:clF}
\end{equation}
\vspace{-0.4in}
\item[\textbf{(A3)}]{ For any eigenvalue $\lambda$ of $A$,
\begin{equation}
\text{Real}(\lambda)<0  \quad \text{and} \quad | 1+\Ascale \lambda | <1
\label{e:Aconds}
\end{equation}
}
\end{romannum} 
It is assumed without loss of generality that   $\Ascale =1$.  

Under Assumptions~\textbf{A1} and \textbf{A2}, the covariance matrix   in \eqref{e:SigmaDelta} can be expressed
\begin{equation}
\Sigma^\Delta = \Expect[ \Delta_{n+1}^* (\Delta_{n+1}^*) ^\transpose ]
\label{e:sigma_delta}
\end{equation}
The noise covariance corresponding to parameter estimate $\theta_n$ is denoted
\begin{equation}
\Sigma^\Delta_{n+1} = \Expect[ \Delta_{n+1} (\Delta_{n+1}) ^\transpose ]
\label{e:sigma_delta_n}
\end{equation}
 Under the assumption that $\lim_{n\to\infty} \theta_n = \theta^*$ in $L_2$ we obtain
 $\Sigma^\Delta_n = \Sigma^\Delta +o(1)$.

Even in the linear setting,   full stability and coupling arguments are not yet available because the assumptions do not ensure that $\haA_n^{-1}\to A^{-1}$ in $L_2$.   Analysis is restricted to the  simplified  SNR and PolSA algorithms, defined as follows:
\begin{eqnarray}  
\hspace{-0.25in} \text{\bf SNR:} \,\,\,\,	\Delta\theta^*_{n+1} = \hspace{-0.15in}&&-  \alpha_{n+1} A^{-1}  f_{n+1}(\theta_n^*)  
		\label{e:SNR_simple}  
		\\
\hspace{-0.25in} \text{\bf PolSA:} \,\,\,\,	\Delta\theta_{n+1} = \hspace{-0.15in}&&[I+     A]  \Delta\theta_n  
		+   \alpha_{n+1}  f_{n+1}(\theta_n)  
				\label{e:PolSA_simple}  
\end{eqnarray} 
For the linear model, the recursion \eqref{e:SNR_simple} becomes
\begin{equation}
\begin{aligned}
\Delta\theta^*_{n+1} 
	& = -  \alpha_{n+1} A^{-1}  [A_{n+1} \theta^*_n - b_{n+1} ]
	=   -\alpha_{n+1}  [ \tiltheta^*_n  - A^{-1} \Delta_{n+1}  ]
\end{aligned} 
\label{e:babySNR}
\end{equation}

The SNR algorithm is in some sense optimal under general conditions.   The proof of \Prop{t:snr}    is contained in Section~\ref{s:babySNR} of the Appendix.

\begin{proposition}
\label{t:snr}
Suppose that Assumptions A1--A3 hold.  Then, the following hold for the 
estimates $\{\theta^\bullet_n \}$ obtained from the
SNR algorithm and $\{\theta^*_n \}$ obtained from
\eqref{e:SNR_simple}: 
\begin{romannum}
\item  The representations hold:
\begin{align}
\tiltheta^\bullet_n  & =  - \haA_n^{-1}\frac{1}{n}  \sum_{k=1}^n \Delta^*_k  \qquad \text{whenever the inverse $\haA_n^{-1} $ exists}
\label{e:SNR_rep1}
\\
\tiltheta^*_n  & =   - A^{-1}\frac{1}{n}  \sum_{k=1}^n \Delta_k
\label{e:SNR_rep2}
 \end{align}
 Consequently, each converges to zero with probability one.

\item  The scaled covariances $\Sigma_n = n \Expect[\tiltheta^*_n(\tiltheta^*_n)^\transpose]$ 
and   $\Sigma_n^{22} = n^2 \Expect[\Delta\theta^*_n (\Delta\theta^*_n )^\transpose]$ satisfy
\begin{align}
\lim_{n\to\infty}
\Sigma_n = \lim_{n\to\infty}
\Sigma_n^{22} = \Sigma^* 
\label{e:SNR_cov_converges}
\end{align}
with $\Sigma^*=A^{-1} \Sigma^\Delta {A^{-1}}^{\transpose}$; the optimal covariance  \eqref{e:SigmaStar}.
\qed
\end{romannum} 
\end{proposition} 

\spm{no space:
Convergence of the scaled covariance for the SNR algorithm is also possible, provided  $\{ \haA_n^{\dagger} \}$ is convergent to $A$ in  $L_2$.  
}
    
A drawback with SNR is the matrix inversion.   The PolSA algorithm  is simpler and enjoys the same attractive properties.   This is established
through coupling:

\begin{proposition}
\label{t:couple}
Suppose assumptions \textbf{(A1)}--\textbf{(A3)} hold.   Let 
$\{\theta_n^*\}$ denote the iterates using SNR \eqref{e:SNR_simple} 
and  $\{\theta_n \}$  the iterates obtained using \eqref{e:PolSA_simple},      with identical initial conditions.     
Then,
\begin{equation}
\sup_{n\ge 0} \,   n^2 \Expect[ \|  \theta_n - \theta_n^* \|^2]  <\infty   
\label{e:couple}
\end{equation} 
Consequently,  the limits \ref{e:SNR_cov_converges}  hold for the PolSA algorithm \eqref{e:PolSA_simple}:
\[
\lim_{n\to\infty}  n \Expect[\tiltheta^*_n(\tiltheta^*_n)^\transpose]
=
\lim_{n\to\infty}
 n^2 \Expect[\Delta\theta^*_n (\Delta\theta^*_n )^\transpose] = \Sigma^* 
\]
\end{proposition}

Other than  SNR and the Polyak-Ruppert averaging technique, to the best of our knowledge, PolSA is the only other known   algorithm that achieves optimal asymptotic variance.

The proof of \Prop{t:couple}, contained in Section~\ref{s:couple} of the supplementary material, is based on a justification of the heuristic used   to construct the deterministic recursion \eqref{e:OurLinearHBdet}.  


An illustration is provided in \Fig{f:200} for the linear model $f_n(\theta) = A\theta + \Noise_n$ in which $-A$ is symmetric and positive definite,  with $\lambdamax(-A)=1$, and $\{\Noise_n\}$ is i.i.d.\ and Gaussian.  Shown are the trajectories of $\{\theta_n(1) : n\le 10^5\}$ (note that $\Sigma^*(1,1)$ is over one million).


\begin{figure*}[t]
	\Ebox{1}{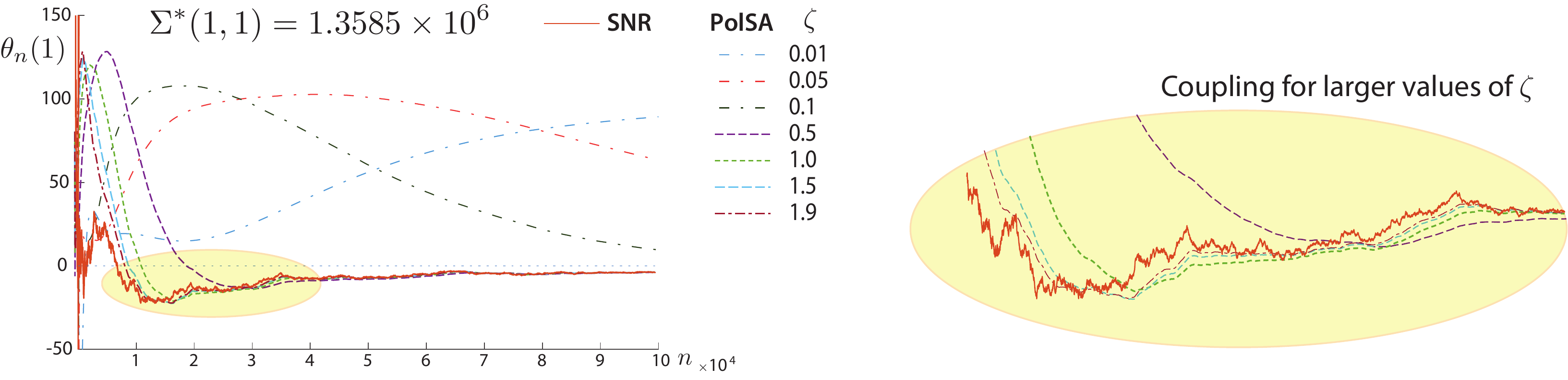} 
\vspace{-1.25em}
\caption{Coupling between PolSA and SNR   occurs quickly for $0.5\le \Ascale\le 1.9$.} 
\vspace{-.5em}
\label{f:200}
\end{figure*}

\subsection{Applications}
\label{s:app}

\subparagraph{Reinforcement learning} 

\Section{s:num} describes application to Q-learning, and includes numerical examples.  
\Section{s:TD_appendix} in the supplementary material contains a full account of TD-learning.  In particular, the LSTD  algorithm can be regarded as an instance of SNR:  \eqref{e:SAa} with $G_n=-\haA_n^{-1}$ an estimate of $G^* = -A^{-1}$.     

\subparagraph{Stochastic optimization} 


A common application of stochastic approximation is convex optimization.    In this setting,  $\barf(\theta) = \nabla \Expect[ J_n(\theta)]$ for a sequence of smooth functions $\{J_n\}$,  and then $f_n= - \nabla J_n$.   
The theory developed in this paper is applicable,  except in degenerate cases.   For comparison, consider the quadratic optimization problem in which
$f_n(\theta)=A\theta -b +\Delta_n$,  with $-A>0$.  The stability condition \eqref{e:Aconds} holds provided   $\Ascale < 1/\lambdamax(-A)$: a condition familiar in the convex optimization literature.

\ad{Does not belong in applications!!
\bf
why not?   }

\subparagraph{Stochastic algorithms for deterministic optimization} 

Finally, we explain why the results of this paper \textit{do not} apply in typical deterministic optimization domains.

It is common to use randomized algorithms to solve deterministic optimization problems. Two examples are ERM and the randomized coordinate descent.   In these and other examples, the algorithms are designed so that randomness vanishes as $\theta$ approaches $\theta^*$.  The asymptotic covariance matrix $\Sigma^\Delta$ is zero, and hence   the asymptotic covariance  \eqref{e:SigmaStar} also vanishes.

\section{Variance analysis of the NeSA algorithm}
\label{s:NeSA}


The NeSA algorithm \eqref{e:NeSA} has a finite asymptotic covariance that can be expressed as the solution to a Lyapunov equation.    We again restrict to the linear model, so that the recursion \eqref{e:NeSA} 
(with $\Ascale = 1$)
becomes 
\begin{equation}
\begin{aligned}
\hspace{-0.085in}\Delta\theta_{n+1} = [I + A_{n+1}]\Delta\theta_n + \alpha_{n+1} [A_{n+1}\theta_n -b_{n+1}]
\end{aligned}
\label{e:accSAn_Zap}
\end{equation}

Stability of the recursion requires a strengthening of assumption \eqref{e:Aconds}.     
Define the linear operator $\clL \colon \Re^{d\times d} \to \Re^{d\times d}$ as follows: For any matrix $Q \in \Re^{d\times d}$,
\begin{equation}
\clL(Q) \eqdef \Expect [(I+A_n) Q (I+A_n)^{\transpose}]
\label{e:clL}
\end{equation}

Define the $2d$-dimensional vector processes  $\Phi_n \eqdef ( \sqrt{n} \tiltheta_n , n \Delta \theta_n)^\transpose$, and
\begin{equation}
\Sigma_n \eqdef \Expect  [  \Phi_n \Phi_n^\transpose]
=
\begin{bmatrix}
\Sigma_n^{11} & \Sigma_n^{12}
\\
\Sigma_n^{21} & \Sigma_n^{22}
\end{bmatrix}
\label{e:SigmanAnBlockForm}
\end{equation} 
The following assumptions are imposed throughout: \spm{where $\clF$ is the natural filtration \eqref{e:clF}.
}
\begin{romannum}
\item[\textbf{(N1)}]  $\{\tilA_n, \tilb_n\}$ are bounded martingale difference sequences.  Moreover, for any 
matrix $Q$,
\[
  \Expect [(I+A_n) Q (I+A_n)^{\transpose}\mid \clF_{n-1}] = \clL(Q) 
\]


\item[\textbf{(N2)}]     The bounds in \eqref{e:Aconds} hold, and  the spectral radius of $\clL$ is strictly bounded by unity.

\item[(N3)] The covariance sequence $\{\Sigma_n\}$ defined in \eqref{e:SigmanAnBlockForm}  is bounded.
\\
In \Section{s:VarNeSA} of the supplementary material we discuss how (N3) can be relaxed.

\end{romannum}



\begin{proposition}
\label{t:OptVarZapHBAn}
Suppose that (N1) and (N2) hold.  
Then, 
\begin{equation}
\lim_{n\to\infty}  \Sigma_n=\begin{bmatrix}
\Sigma_\infty^{11} & 0
\\
0 & \Sigma_\infty^{22}
\end{bmatrix}
\label{e:SigLimAn}
\end{equation}
in which  the second limit is the  solution to the Lyapunov equation
\begin{equation} 
\Sigma_\infty^{22} = \clL(\Sigma_\infty^{22}) + \Sigma^{\Noise}
\label{e:SigLim22An}
\end{equation}
(an explicit solution is given in eqn.~\eqref{e:SigLim22An_inverse} 
of the supplementary material),
and
\begin{equation}
\Sigma_\infty^{11}= - \Sigma_\infty^{22} - A^{-1} \Sigma_\infty^{22} - \Sigma_\infty^{22} A^{-1}
\label{e:LimSigma11An}
\end{equation}
\end{proposition}
\spm{AD (writing to you Thurs eve):  I am super suspicious of $\Sigma_\infty^{11}$!   At some point we need to prove that this is larger than $\Sigma^*$.
}

The following result is a corollary to \Prop{t:OptVarZapHBAn},
with an independent proof provided in \Section{s:VarAnaZapHB} of the supplementary material.

\begin{proposition}
\label{t:OptVarZapHB}
Under (N1)--(N3) the conclusions of \Prop{t:OptVarZapHBAn}
hold for the PolSA recursion \eqref{e:PolSA_simple}.  In this case the solution 
to the Lyapunov equation is the optimal covariance: 
\begin{equation}
\Sigma_\infty^{11}= \Sigma^* \eqdef A^{-1}  \Sigma_\Delta {A^{-1} }^\transpose
\label{e:LimSigma11}
\end{equation}
and $\Sigma_\infty^{22}\ge 0$ is the unique   solution to the Lyapunov equation
\begin{equation} 
\Sigma_\infty^{22} = (I+A) \Sigma_\infty^{22} (I+A)^\transpose + \Sigma^{\Noise}
\label{e:SigLim22}
\end{equation} 
\end{proposition}


The proofs of the following   are contained in Section~\ref{s:VarNeSA} of the supplementary material.

\begin{lemma}
\label{t:OffDiagSigmaAn}
The following approximations hold,  with    $\psi_n \eqdef \sqrt{n}  \Sigma_n^{21}$:
\begin{equation} 
\begin{aligned}
\Sigma_{n+1}^{22}   &=   \clL(\Sigma_n^{22}) + \Sigma^{\Noise}
+ o(1)
\,,
\qquad
\psi_n    = -\Sigma_n^{11}  - A^{-1}  \Sigma_\infty^{22}    + o(1)\,,\qquad n\ge 1\, .
\end{aligned}
\label{e:OffDiagSigmaAn}
\end{equation}
\end{lemma}


The second iteration is used together with the following result to obtain   \eqref{e:LimSigma11An}.

\begin{lemma}
\label{t:SAforCovarAn}
The following approximation holds:
\begin{equation}
\Sigma_{n+1}^{11}  = 
\Sigma_n^{11}   + \alpha_{n+1} \Big(  \Sigma_n^{11}   + A \Sigma_n^{11} 
+ \Sigma_n^{11}A^\transpose 
 + \psi_n ^\transpose  (I +A)^\transpose + {(I+A) \psi_n }  
  +  \clL(\Sigma^{22}_\infty) + \Sigma^{\Noise}  + o(1)\Big)
\label{e:Sigma11An}
\end{equation}
\end{lemma}

\subparagraph{Proof of \Prop{t:OptVarZapHBAn}:  }


The first approximation in \eqref{e:OffDiagSigmaAn} combined with
(N2) implies that the sequence 
$\{\Sigma_n^{22}\}$ is convergent, and the limit is the solution to the fixed point equation \eqref{e:SigLim22An} (details are provided in  Section~\ref{s:recursions} of the supplementary material).

Substituting the approximation \eqref{e:OffDiagSigmaAn} for $\psi_n$ into 
\eqref{e:Sigma11An} and simplifying gives
\[
\begin{aligned}
\Sigma_{n+1}^{11}  = \Sigma_n^{11} + \alpha_{n+1} \Big( - & \Sigma_n^{11} - \Sigma_\infty^{22} 
 - A^{-1} \Sigma_\infty^{22} - \Sigma_\infty^{22} A^{-1}  + o(1)\Big)
\end{aligned}
\]
This can be regarded as a Euler approximation to the ODE:
\[
\ddt x_t = -x_t  - \Sigma_\infty^{22} - A^{-1} \Sigma_\infty^{22} - \Sigma_\infty^{22} A^{-1}
\]
Stochastic approximation theory can be applied to establish that the limits of $\{\Sigma_n^{11} \}$ and $\{ x_t\}$ coincide with the stationary point \cite{bor08a},
which is \eqref{e:LimSigma11An}.
\qed

\begin{figure*}

\Ebox{.9}{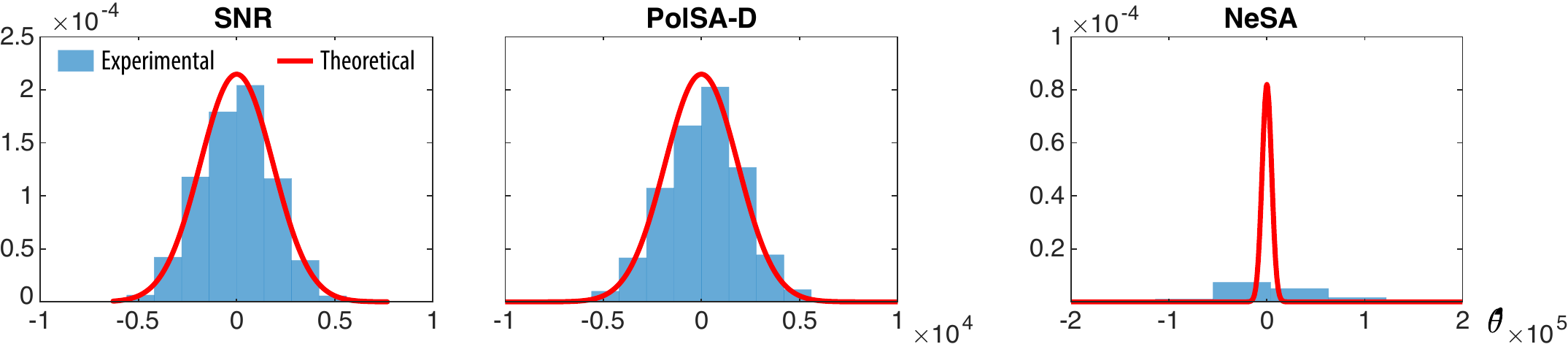} 
\vspace{-1.5em}
\caption{Histograms for entry 18 of $\{\sqrt{n}\tiltheta_n\}$ for three algorithms at iteration $10^6$.} 
\label{f:18hist}
 \vspace{-.75em}

\end{figure*}

\vspace{-0.2in}

\section{Application to Q-learning}
\label{s:num}

Consider a discounted cost MDP model with state space $\state$, action space   $\U$,   cost function $c\colon\state\times\U\to\Re$, and discount factor $\beta\in(0,1)$.  It is assumed that the state and action space are finite:
denote $\nd = |\state|$,  $\ninp = |\U|$, and $P_u$ the $\ell \times \ell$  controlled transition probability matrix.



The Q-function  is the solution to the Bellman equation:
\begin{equation}  
Q^*(x,u) = c(x,u) + \beta \Expect [ \min_{u'} Q^*(X_{n+1}, u') \mid X_n = x \,, U_n = u ]
\label{e:BE}
\end{equation}
The goal of Q-learning is to learn an approximation to $Q^*$.
Given $d$ basis functions
$\{\phi_i: 1 \leq i \leq d\}$, with each $\phi_i: \state \times \U \to \Re$, 
and a parameter vector $\theta\in\Re^d$, the   Q-function estimate is denoted
$
Q^\theta(x,u) = \theta^\transpose \phi(x,u)$,
and its minimum,  $\underline{Q}^{\theta} (x) \eqdef \displaystyle \min_{u'} Q^\theta(x, u)$.
 
Watkins' Q-learning algorithm is designed to compute the exact Q-function that solves the Bellman equation \eqref{e:BE} (\cite{wat89,watday92a}).  In this setting, the basis is taken to be the set of indicator functions:  $\phi_i(x,u) = \ind\{x=x^i, u=u^i\}$,  $ 1\le i\le d $, with $d=|\state\times\U|$. {The goal is to find $\theta^* \in \Re^d$ such that $\barf(\theta^*) = 0$, where, for any $\theta \in \Re^d$,
\begin{equation*}
\begin{aligned}
\barf(\theta)  & = \Expect \Big [\phi(X_n,U_n) \Big(   c(X_n,U_n) + \beta  \underline{Q}^{\theta} (X_{n+1})  - Q^{\theta}(X_n,U_n) \Big)   \Big  ]   
\end{aligned} 
\label{e:WatkinsGoal}
\end{equation*}
where the  
expectation is with respect to the steady state distribution of the Markov chain.

The basic algorithm of Watkins can be  written  
\begin{equation}
\Delta \theta_{n+1} =    \alpha_{n+1} \haD_{n+1} \bigl[A_{n+1}\theta_n - b_{n+1} \bigr]
\label{e:linearSA}
\end{equation} 
in which the  matrix gain is diagonal,  with 
$\haD_n(i,i) ^{-1} =  \frac{1}{n} \sum_{k=0}^{n-1} \ind\{ (X_k,U_k) = (x^i,u^i) \} $,  
and with $\pi_n(x) \eqdef \displaystyle \argmin_u Q^{\theta_n}(x,u)$, 
\[
\begin{aligned}
\hspace{-0.065in}A_{n+1}  & =        \phi(X_n,U_n) \{  \beta    \phi (X_{n+1}, \pi_{n}(X_{n+1} ) ) - \phi(X_n,U_n)  \bigr\}^\transpose
\\
b_{n+1} & = c(X_n,U_n)  \phi(X_n,U_n)
\end{aligned}
\]
See  (\cite{csa10}) for more details. 
Among the other algorithms compared are 
\[
\begin{aligned}
& \text{\bf SNR:}  \hspace{0.45in} \Delta \theta_{n+1} =    -\alpha_{n+1} \haA_{n+1}^{-1} \bigl[A_{n+1}\theta_n - b_{n+1} \bigr]
\\ 
& \text{\bf PolSA:}  \hspace{0.35in} \Delta \theta_{n+1} =    [I+      \haA_{n+1}]  \Delta\theta_n  + \alpha_{n+1}   \bigl[A_{n+1}\theta_n - b_{n+1} \bigr]
\\ 
& \text{\bf PolSA-D:}  \hspace{0.2in} \Delta \theta_{n+1} =    [I+      \haD_{n+1} \haA_{n+1}]  \Delta\theta_n  
+ \alpha_{n+1} \haD_{n+1}  \bigl[A_{n+1}\theta_n - b_{n+1} \bigr]
\\
&
  \text{\bf NeSA:} 
  \hspace{0.4in} \Delta \theta_{n+1} =    [I+  A_{n+1}]  \Delta\theta_n  
 				+ \alpha_{n+1} \bigl[A_{n+1}\theta_n - b_{n+1} \bigr] 
\end{aligned}
\]
In each of these algorithms, \eqref{e:haAn} is used to recursively estimate $\haA_{n}$.
We have taken $\Ascale =1$ in PolSA.   The variant PolSA-D is \eqref{e:PolSAa} with $ G_{n+1} = \haD_{n+1} $,
and $ \Mtum_{n+1} $ chosen so that   coupling with SNR can be expected.

\spm{no space:, with $A_{n+1}$ defined above.
}
%

The SNR algorithm considered coincides with the \textit{Zap Q-learning} algorithm of \cite{devmey17b,devmey17a}.  
A simple 6-state MDP model was considered in this prior work, with the objective of finding the stochastic shortest path. \Fig{f:18hist} contains histograms of $\{\sqrt{n}\tiltheta_n\}$ obtained from $1000$ parallel simulations  of PolSA-D, SNR and NeSA algorithms for this problem. It is observed that the histograms of PolSA-D and SNR nearly coincide after $n=10^6$ iterations (performance for PolSA is similar). The histogram for NeSA shows a much higher variance, but the algorithm requires by-far the least computation per iteration. This is specifically true for Watkins' Q-learning since $A_{n+1}$ is a sparse matrix,  with just $2$ non-zero entries.

\begin{figure*}[t]
\vspace{.1em}
\Ebox{1}{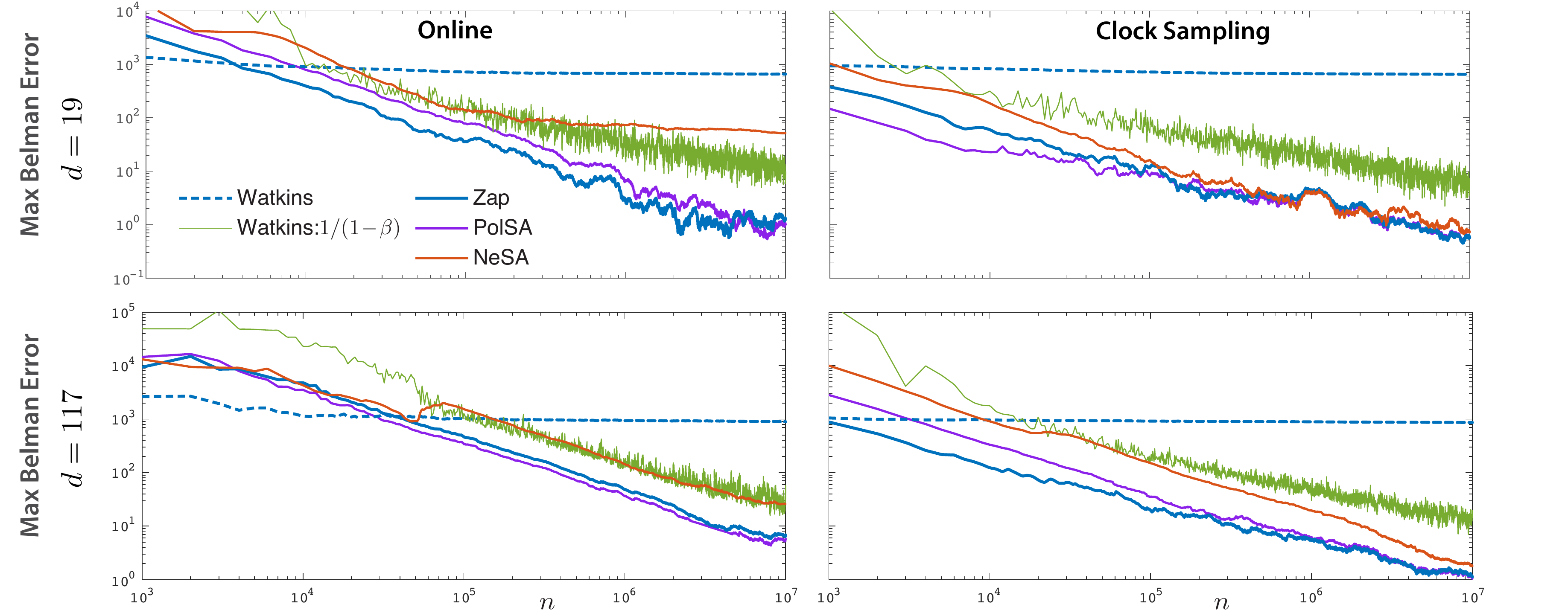}
\vspace{-1.5em}
\caption{Bellman error for  $n\le 10^7$ in the shortest path problem with  $d=19$ and $d = 117$.   Deterministic exploration leads to much faster convergence in the NeSA algorithm.} 
\label{f:stoch+det_25_015}
\vspace{-.75em}
\end{figure*}

Experiments were also performed for larger examples.   Results from two such experiments are shown in
\Fig{f:stoch+det_25_015}.     The MDP model is once again a stochastic shortest path problem. The model construction was based on the creation of a graph with $N$ nodes,  in which the probability of an edge between a pair of nodes is i.i.d.\ with probability $p$.   Additional edges $(i,i+1)$ are added, for each $i<N$, to ensure the resulting graph is  strongly connected.
 
The transition law is similar to that used in the  finite state-action example of \cite{devmey17b}: with probability $0.8$ the agent moves in the desired direction, and with remaining probability it ends up in one of the neighboring nodes, chosen uniformly.
Two exploration rules were considered:  the ``online'' version wherein at each iteration the agent randomly selects a feasible action (also known as asynchronous Q-learning),  and the offline ``clock sampling'' approach in which state-action pairs $(x^i,u^i)$ are chosen sequentially (also known as synchronous Q-learning). In the latter, at stage $n$, if $(x,u)$ is the current state-action pair,  a random variable $X_{n+1}'$ is chosen according to the distribution $P_u(x,\varble)$,   and the $(x,u) $ entry of the Q-function is updated according to the particular algorithm using the triple $(x,u,X_{n+1}')$.   A significant change to Watkins' iteration \eqref{e:linearSA} in the synchronous setting is that $\haD_n$ is replaced by $d^{-1} I$ (since each state is visited the same number of times after each cycle).   This combined with deterministic sampling is observed to result in significant variance reduction.   
The \textit{synchronous speedy Q-learning} recursion of \cite{azamunghakap11} appears similar to the NeSA algorithm with clock sampling. 

Two graphs were used in the survey experiments, one resulting in an MDP  with  $d = 19$ state-action pairs and another resulting in a larger MDP with $d = 117$ state-action pairs.  The plots in \Fig{f:stoch+det_25_015} show Bellman error as a function of iteration $n$ for the two cases (for definitions see \cite{bertsi96a,devmey17a}). Comparison of the performance of algorithms in a deterministic exploration setting versus the online setting is also shown. The coupling of PolSA and the Zap algorithms are easily observed in the clock sampling case. In the online case, it is less prominent, but can still be seen when $d = 19$.



%
%


\section{Conclusions}

It is exciting to see how the intuitive transformation from SNR to PolSA and NeSA can be justified theoretically and in simulations.     While the covariance of NeSA is not optimal,  it is the simplest of the three algorithms and,
 performs well in  applications to Q-learning.


An important next step is to create adaptive techniques to ensure fast coupling or other ways to ensure fast forgetting of the initial condition.   It is possible that techniques in \cite{jaikakkidnetsid17}  may be adapted. 
The work can be extended in several ways:
\begin{romannum}
\item It will be of great interest to pursue analysis of the proposed algorithms in the special case of nonlinear optimization. It is possible that the structure of the problem such as convexity of the objective and smoothness of the gradients could help us derive bounds on the transients.

\item The authors in \cite{devmey17b} suggest that their algorithm can be used for Q-learning with function approximation. It would be interesting to see how the PolSA and NeSA algorithms can be extended to this setting. Applications to TD-learning with function approximation is discussed in \Section{s:TD_appendix} of the Appendix.
\end{romannum}


 \clearpage

{
\bibliography{strings,markov,q,ExtrasNIPS18}
}

\appendix

\clearpage

\begin{center}
	{\Large \bf
Appendix}	
\end{center}
%
%
%
%
%
%
%
%
%
%
%
%
%
%
%
%
%
%
%
%
%
%
%
%
%

\section{Proof of \Prop{t:snr}}
\label{s:babySNR}

\textbf{Proof of \Prop{t:snr}}
The proof of the first limit in \eqref{e:SNR_cov_converges} is obtained through Taylor series arguments surveyed in \cite{devmey17a}.  
The second limit in \eqref{e:SNR_cov_converges} follows from the first, and the representation
\[
(n+1)\Delta\tiltheta^*_{n+1}  =   -  [ \tiltheta^*_n  - A^{-1} \Delta_{n+1}  ]
\]
This representation uses \eqref{e:babySNR} and the assumption $\alpha_{n+1}=1/(n+1)$. Consequently,
under the martingale difference property for $\{\Delta_n\}$,
\[
\Sigma_{n+1}^{22} = \frac{1}{n} \Sigma_n  + A^{-1} \Sigma^\Delta_{n+1} (A^{-1})^\transpose   
\]
The first term vanishes because $\{\Sigma_n\}$ is convergent.    Moreover, because $\Delta_{n+1} = \Delta_{n+1}^* + \tilA_{n+1} \tiltheta_n^*$, it then follows from     that $\Sigma^\Delta_{n+1} = \Sigma^\Delta + o(1)$,  and hence
\[
\Sigma_{n+1}^{22} = A^{-1} \Sigma^\Delta  (A^{-1})^\transpose   + o(1)
=   \Sigma^* + o(1).
\]

The proof of \eqref{e:SNR_rep1} is obtained as follows:  it follows from the definition that 
\[
 \tiltheta^\bullet_{n+1} 
	 = 
	  \tiltheta^\bullet_{n} 
	 -  \alpha_{n+1} \haA_{n+1}^{\dagger}  [A_{n+1} \tiltheta^\bullet_n + \Delta^*_{n+1} ]
\]
and also
\[
(n+1) \haA_{n+1} = 
n \haA_{n} +A_{n+1} 
\]
Consequently, when $  \haA_{n+1} ^{-1}$ exists,    so that $ \haA_{n+1}^{\dagger}= \haA_{n+1} ^{-1}$,
\[
\begin{aligned}
(n+1) \haA_{n+1} \tiltheta^\bullet_{n+1}  & =  [n \haA_{n} +A_{n+1} ]  \tiltheta^\bullet_{n} 
-  [A_{n+1} \tiltheta^\bullet_n + \Delta^*_{n+1} ]
\\
&=n \haA_{n}  \tiltheta^\bullet_{n} -\Delta^*_{n+1}  
\end{aligned} 
\]
Summing this telescoping series gives  \eqref{e:SNR_rep1}.   

The proof of  \eqref{e:SNR_rep2} is similar and simpler, since  we immediately obtain from the definition \eqref{e:babySNR},
\[
(n+1) A \tiltheta^*_{n+1}    =  n  A  \tiltheta^\bullet_{n} -\Delta_{n+1}  
\]
Summing each side then gives  \eqref{e:SNR_rep2}.
\qed

\section{Coupling}
\label{s:couple}

We present here the proof of \Prop{t:couple}, based on a transformation of SNR so that it resembles PolSA with a vanishing disturbance sequence.  This is essentially a reversal of the manipulations applied to derive \eqref{e:OurLinearHBdet} from an approximation of \eqref{e:HBmatrix} at the start of \Section{s:momentum}, but now in  a stochastic setting.

\ad{\rd{Why ``approximation" of \eqref{e:HBmatrix}?? And I feel like referring the reader to the deterministic algorithm might be a distraction - but I'll try to clarify.}}

Consider the recursion \eqref{e:PolSA_simple}. For simplicity we take $\Ascale = 1$ (this is without loss of generality by re-defining the matrix $A$).

It is simplest to first prove \Prop{t:couple} when $\{A_n\}$ is deterministic:  $\tilA_n\equiv 0$. The proof of the stochastic case is presented next.

\subsection{Deterministic matrix sequence}  

\begin{lemma}
	\label{t:couplePolSA}
	The SNR and PolSA recursions \eqref{e:SNR_simple}--\eqref{e:PolSA_simple} with deterministic $\{A_n\}$ ($A_n \equiv A$) can be expressed, respectively
	\begin{eqnarray}  
	\Delta\theta^*_{n+1} &=&   [I+      A]  \Delta\theta_n^* 
	+   \alpha_{n+1}  [A \tiltheta_n^* + \Noise_{n+1} ]  +\clE_{n+1}
	\label{e:SNR_det}  
	\\
	\Delta\theta_{n+1} &=&   [I+     A]  \Delta\theta_n  
	+   \alpha_{n+1}  [A \tiltheta_n + \Noise_{n+1} ]\,,
	\label{e:PolSA_det}  
	\end{eqnarray} 
	where $\Noise_n = b - b_n$ and
	\begin{equation}
	\clE_{n+1} = [I+A]  \big ( \Delta\theta^*_{n+1}  -  \Delta\theta^*_n  \big ) \,,  \qquad n\ge 0\,.
	\label{e:clE}
	\end{equation} 
\end{lemma}

\begin{proof}
	Recall the simplified PolSA recursion \eqref{e:PolSA_simple}: 
	\[
	\Delta\theta_{n+1} = [I+     A]  \Delta\theta_n  
	+   \alpha_{n+1}  f_{n+1}(\theta_n)  
	\]
	Substituting the linear model \eqref{e:fLin} into the above recursion, \eqref{e:PolSA_det}  is obtained. Furthermore, in the special setting $A_n \equiv A$, the noise sequence $\Delta_n$ in \eqref{e:SigmaNoise} becomes $\Delta_n = \Delta_n^* = b - b_n$.
	
	Next, recall the SNR recursion \eqref{e:SNR_simple}:
	\[
	\Delta\theta^*_{n+1} = -  \alpha_{n+1} A^{-1}  f_{n+1}(\theta_n^*)  
	\]
	Multiplying both sides of the above recursion by $A$, and substituting for $f_{n+1}$ the linear model \eqref{e:fLin} gives
	\[
	\begin{aligned}
	0  &= A\Delta \theta_{n+1}^* + \alpha_{n+1}  [A \tiltheta_n^* + \Noise_{n+1}   ]
	\\
	&=  A\Delta \theta_n^* + \alpha_{n+1}  [A \tiltheta_n^* + \Noise_{n+1}]  -\big(  \Delta\theta^*_{n+1}  -  \Delta\theta^*_n \big ) + \clE_{n+1} 
	\end{aligned}
	\]
	where once again $\Delta_n = b - b_n$, and $\clE_{n+1}$ is defined in \eqref{e:clE}.
	Moving  $ \Delta\theta^*_{n+1} $ to the left-hand side completes the proof of \eqref{e:SNR_det}.
\end{proof}

Denote:
\begin{equation}
\bartheta_n \eqdef \theta_n -\theta_n^* \qquad \xi_n \eqdef n\bartheta_n
\label{e:bartheta_and_xi}
\end{equation}
The proof of \Prop{t:couple} requires that we establish uniform bounds on each of these sequences.    

Denoting $\Delta\bartheta_n \eqdef \bartheta_n - \bartheta_{n-1}$, the following lemma establishes a recursion for $\{\bartheta_n : n \geq 0 \}$ that is similar to the PolSA recursion \eqref{e:PolSA_det}:
\begin{lemma}
	\label{t:bartheta_coupling}
	The error sequence $\{\bartheta_n : n \geq 0 \}$ evolves according to the recursion
	\begin{equation}
	\Delta\bartheta_{n+1} =   [I+     A]  \Delta\bartheta_n  
	+   \alpha_{n+1} A \bartheta_n  -\clE_{n+1}
	\label{e:bartheta_coupling}
	\end{equation}
	in which the  sequence  $\{ \clE_n : n \geq 1\}$ defined in \eqref{e:clE} satisfies the following for a constant $b_\clE< \infty$:
	\begin{romannum}
		\item $\{n \clE_n : n\ge 1\}$ is a bounded sequence: $\| n  \clE_n \| \leq b_{\clE}$ for all $n \geq 1$.
		\item  Its partial sums are also bounded: Defining
		\[
		S_n^\clE \eqdef \sum_{k=1}^n k\clE_k\,, \quad n\ge 1\, ,
		\]
		$\| S_n^\clE \| \leq b_\clE$ for all $n \geq 1$.
	\end{romannum}
\end{lemma}

\begin{proof}
	Representation \eqref{e:bartheta_coupling} directly follows by subtracting \eqref{e:SNR_det} from \eqref{e:PolSA_det}. We now prove that the error sequence $\{n \clE_n : n\ge 1\}$ satisfies the properties in (i) and (ii).
	
	Recalling that $\alpha_{n+1} = 1/(n+1)$ in the SNR recursion \eqref{e:SNR_simple} we have: 
	\begin{equation}
	\Delta\theta_{n+1}^* = - (n+1)^{-1}(\theta_n^* + A^{-1}\Noise_{n+1}).
	\label{e:SNR_delta_theta_n}
	\end{equation}
	Consequently,
	\begin{equation}
	\begin{aligned}
	(n+1)  \big(  \Delta\theta^*_{n+1}  -  \Delta\theta^*_n  \big) 
	&=  (n+1)     \Delta\theta^*_{n+1}  -  n \Delta\theta^*_n   - \Delta\theta^*_n  
	\\
	& =  - (\theta_n^* + A^{-1}\Noise_{n+1}) + (\theta_{n-1}^* + A^{-1}\Noise_{n})  - \Delta\theta^*_n  
	\\
	&
	=- 2\Delta\theta^*_n   - A^{-1}(\Noise_{n+1} - \Noise_n)
	\end{aligned}
	\label{e:l_clE}
	\end{equation}
	where the second equality follows from \eqref{e:SNR_delta_theta_n}, and the last equality is obtained by combining the common terms.  Under the assumption of the \Prop{t:couple}, the sequences $\{\theta^*_{n} : n \geq 0\}$ and $\{\Delta_{n} : n \geq 0\}$ are bounded. Therefore, the right hand side of \eqref{e:l_clE} is also bounded, and multiplying both sides of the equation by $[I + A]$, part (i) of the lemma follows.
	It is also easy to see that the right hand side of \eqref{e:l_clE} is a telescoping sequence. Therefore, for each $n \geq 1$,
		\[
		\begin{aligned}
		\sum_{k=1}^n k \big( \Delta\theta^*_{k}  -  \Delta\theta^*_{k - 1} \big )   & = \sum_{k=1}^n \big( - 2\Delta\theta^*_{k - 1}   - A^{-1}(\Noise_{k} - \Noise_{k - 1}) \big )
		\\
		& = - 2 \theta^*_{n - 1} + 2 \theta_0^*  - A^{-1}(\Noise_{n} - \Noise_{0})
		\end{aligned}
		\]
		Once again, the right hand side of the above equation is uniformly bounded in $n$ under the assumptions of \Prop{t:couple}.  
			\ad{All this in colors because I don't know what to cite for boundedness of $\theta_n^*$! - which is the main thing used in the proof.
			\\
			sm:  I added consistency to the first   proposition,  so theta-n* is bounded}
\end{proof}

\begin{lemma}
	\label{t:bartheta_coupling2}
	The normalized error sequence $\{\xi_n: n \geq 0\}$ defined in \eqref{e:bartheta_and_xi} satisfies the following recursion:
	\begin{eqnarray}  
	\Delta\xi_{n+1} = \hspace{-0.15in}&&[I+A]  \Bigl(  \Delta\xi_n +2 [ \bartheta_n -  \bartheta_{n-1} ]   \Bigr)  
	-(n+1)\clE_{n+1} 
	\label{e:Delta_xi_1}
	\end{eqnarray} 
	where $\Delta \xi_n \eqdef \xi_n - \xi_{n-1}$.
	\qed
\end{lemma}
\begin{proof}
	By definition of $\xi_n$ in \eqref{e:bartheta_and_xi}, we have,
	\begin{equation}
	\begin{aligned}
	\Delta \xi_{n+1} 
	& = \xi_{n+1} - \xi_n  
	\\
	& = (n+1) \bartheta_{n+1} - n \bartheta_n
	\\
	& = (n+1) ( \theta_{n+1} - \theta^*_{n+1} ) - n ( \theta_n - \theta_n^* )  
	\\
	& = (n+1) \Delta \theta_{n+1} - (n + 1) \Delta \theta_{n+1}^* +  \theta_n - \theta_n^* 
	\\
	& = (n+1) \Delta \bartheta_{n+1}  + \bartheta_n
	\end{aligned}
	\end{equation}
	Substituting for $\Delta \bartheta_{n+1}$ using \eqref{e:bartheta_coupling} in the above equation, we obtain:
	\begin{equation}
	\begin{aligned}
	\Delta \xi_{n+1} 
	& = (n+1) [I + A] \Delta \bartheta_{n}  + A \bartheta_n + \bartheta_{n} - (n+1) \clE_{n+1}
	\\
	& = [I + A] \Big(  n \bartheta_n - (n - 1) \bartheta_{n-1} + \bartheta_n - 2 \bartheta_{n-1} + \bartheta_{n} \Big) - (n+1) \clE_{n+1}
	\\
	& = [I + A] \Big(  \Delta \xi_n  +  2 \Delta \bartheta_{n} \Big) - (n+1) \clE_{n+1}
	\end{aligned}
	\end{equation}
\end{proof}

We are now ready to provbe \Prop{t:couple} for the deterministic case.

\subparagraph{Proof of \Prop{t:couple} -- deterministic case:}
	On summing each side of the identity \eqref{e:Delta_xi_1} in \Lemma{t:bartheta_coupling2}
	we obtain, for any {$n > m\ge 2$},
	\[
	\begin{aligned}
	\xi_{n+1} -  \xi_{m}&= [I+A]  \Bigl(  \xi_{n} -  \xi_{m-1} +2 [ \bartheta_n -  \bartheta_{m-1} ]   \Bigr)  
	-  \sum_{k=m}^n  (k+1)\clE_{k+1}  
	\end{aligned}
	\]
	Using the definition $\bartheta_n = \xi_n / n$ then gives
	\[
	\xi_{n+1} =
	[I+A]\Bigl( (1+2n^{-1} )   \xi_n  \Bigr)  + \xi_{m}  -
	(1 + 2(m-1)^{-1}) [I+A]    \xi_{m-1}     -  \sum_{k=m}^n  (k+1)\clE_{k+1}
	\]
	
	{Letting m = 2 in the above recursion,}
	\begin{equation}
	\xi_{n+1} =
	[I+A]\Bigl( (1+2n^{-1} )   \xi_n  \Bigr)  +  b_{n+1}^\xi
	\label{e:xi_recursion}
	\end{equation}
	where the final term $ b_{n+1}^\xi$ is bounded in $n$ (part (ii) of \Lemma{t:bartheta_coupling}):
	\[
	b_{n+1}^\xi =  \xi_{2}  -
	3 [I+A]    \xi_{1}     -  \sum_{k=2}^n  (k+1)\clE_{k+1}  
	\]

	Since by assumption, all eigenvalues of $[I+A]$ satisfy $\lambda\big( I + A\big) < 1$, the recursion \eqref{e:xi_recursion} can be viewed as a stable linear system with bounded input $\{b_{n}^\xi\}$. Therefore, for each $Q>0$, there exists a matrix $M > 0$ satisfying the discrete time Lyapunov equation \cite{kai80lin}:
	\begin{equation}
	[I + A]^\transpose M [I + A] = M - Q
	\label{e:lyap}
	\end{equation}
	Choosing $Q = I$, and noting that $I \geq \epsy M$ for some $\epsy > 0 $ that is small enough, we have:
	\begin{equation}
	[I + A]^\transpose M [I + A] \leq \delta^2 M 
	\label{e:lyap_epsy}
	\end{equation}
	where $\delta^2 = 1 - \epsy < 1$.
	Denote $\| \cdot \|_M$ to be the weighted norm with respect to the matrix $M$ that satisfies \eqref{e:lyap_epsy}: For all $\xi \in \Re^d$,
	\begin{equation}
	\|\xi \|_M \eqdef \big( \xi^\transpose M \xi  \big)^{\half}
	\label{e:weighted_norm_M}
	\end{equation}
	Applying the triangle inequality to \eqref{e:xi_recursion}  gives
	\begin{equation}
	\begin{aligned}
	\| \xi_{n+1} \|_M 
	& \leq
	(1+2n^{-1} ) \cdot  \| [I+A]    \xi_n  \|_M  +  \| b_{n+1}^\xi \|_M
	\\
	& \leq
	\delta (1+2n^{-1} ) \cdot \|   \xi_n  \|_M  +  \| b_{n+1}^\xi \|_M
	\label{e:xi_recursion_norm}
	\end{aligned}
	\end{equation}
	Choosing $N_0$ large enough, such that $( 1+2n^{-1} ) \delta < \sqrt{\delta}$ for all $n \geq N_0$,  we have
	\begin{equation}
	\begin{aligned}
	\| \xi_{n+1} \|_M  
	& \leq
	\sqrt{\delta} \cdot \|  \xi_n  \|_M   +  \| b_{n+1}^\xi \|_M \,, \qquad n \geq N_0
	\label{e:xi_recursion_norm_2}
	\end{aligned}
\end{equation}
Consequently, for each $k\ge 1$,   
	\begin{equation}
	\begin{aligned}
	\| \xi_{N_0+k} \|_M  
	& \leq
	{\delta}^{\frac{k}{2}} \cdot \|  \xi_{N_0}  \|_M   + \sum_{i = 0}^{k - 1}  {\delta}^{i/2} \| b_{N_0 + k - i+1}^\xi \|_M 
	\\
	& \leq
	{\delta}^{\frac{k}{2}} \cdot \|  \xi_{N_0}  \|_M   +  \Big( \sum_{i = 0}^{\infty} \delta^{i/2} \Big) \sup_{k \geq 0}  \| b_{n_0 + k}^\xi \|_M   < \infty
	\label{e:xi_recursion_bdd}
	\end{aligned}
	\end{equation} 
	 \qed

\subsection{Proof of \Prop{t:couple} for the general linear algorithm}  

The major difference in the case of random $\{A_n\}$ is that the identity 
\eqref{e:bartheta_coupling} holds with a modified error sequence:
\begin{equation}
\Delta\bartheta_{n+1} =   [I+     A]  \Delta\bartheta_n  
+   \alpha_{n+1} A \bartheta_n  -\clE_{n+1}  + \alpha_{n+1} \tilA_{n+1}\bartheta_n \,,
\label{e:bartheta_coupling_An}
\end{equation}
where the last term appears because we are replacing the last but one term $\alpha_{n+1} A \bartheta_n$ in \eqref{e:bartheta_coupling} with $\alpha_{n+1} A_{n+1} \bartheta_n$.  The error sequence $\{\clE_{n}\}$ is identical to the determinsitic case, and therefore satisfies the properties in \Lemma{t:bartheta_coupling}.
\Lemma{t:bartheta_coupling2} however must be modified due to the additional term (the proof follows exactly the same lines):
\begin{lemma}
	\label{t:bartheta_coupling2_An}
	For the general linear algorithm, the normalized error sequence $\{\xi_n: n \geq 0\}$ defined in \eqref{e:bartheta_and_xi} satisfies the following recursion:
	\begin{eqnarray}  
	\Delta\xi_{n+1} = \hspace{-0.15in}&&[I+A]  \Bigl(  \Delta\xi_n +2 [ \bartheta_n -  \bartheta_{n-1} ]   \Bigr)  
	-(n+1)\clE_{n+1} + \tilA_{n+1}\bartheta_n
	\label{e:Delta_xi_1_An}
	\end{eqnarray} 
	where $\Delta \xi_n \eqdef \xi_n - \xi_{n-1}$.
	\qed
\end{lemma}
The proof then proceeds as in the previous deterministic setting, except that we have to deal with the additional martingale difference sequence   $\{ \tilA_{n+1}\bartheta_n \}$.

\subparagraph{Proof of \Prop{t:couple} -- general linear algorithm}
	On summing each side of the identity \eqref{e:Delta_xi_1_An} in \Lemma{t:bartheta_coupling2_An}
	we obtain, for any {$n > m\ge 2$},
	\[
	\begin{aligned}
	\xi_{n+1} -  \xi_{m}&= [I+A]  \Bigl(  \xi_{n} -  \xi_{m-1} +2 [ \bartheta_n -  \bartheta_{m-1} ]   \Bigr)  
	-  \sum_{k=m}^n  (k+1)\clE_{k+1}  + \sum_{k=m}^n \tilA_{k+1}\bartheta_k
	\end{aligned}
	\]
	Using the definition $\bartheta_n = \xi_n / n$ then gives
 \[
 \begin{aligned} 
	\xi_{n+1} &=
	[I+A]\Bigl( (1+2n^{-1} )   \xi_n  \Bigr)  + \xi_{m}  -
	(1 + 2(m-1)^{-1}) [I+A]    \xi_{m-1}    
	\\&
	\qquad\qquad  -  \sum_{k=m}^n  (k+1)\clE_{k+1} + \sum_{k=m}^n \frac{1}{k} \tilA_{k+1}  \xi_k
\end{aligned} 
\]
	
	The above recursion can be rewritten as,
	\begin{equation}
	\xi_{n+1} =
	[I+A]\Bigl( (1+2n^{-1} )   \xi_n  \Bigr)  +  b_{m, n+1}^\xi + U_{m, n}^{\tilA}
	\label{e:xi_recursion_An}
	\end{equation}
	where:
	\[
	\begin{aligned}
	b_{m, n+1}^\xi & \eqdef  \xi_{m}  -
	(1 + 2(m-1)^{-1}) [I+A]    \xi_{m-1}    -  \sum_{k=m}^n  (k+1)\clE_{k+1}  
	\\
	U_{m, n}^{\tilA} & \eqdef \sum_{k=m}^n \frac{1}{k} \tilA_{k+1}  \xi_k
	\end{aligned}
	\]

\ad{\rd{TYPO HERE??}}
The proof is now similar to the case $A_n\equiv A$, except that we now a weighted $L_2$ norm to obtain $L_2$ bounds.    Let $M$ denote a solution to 	  \eqref{e:lyap_epsy} with $\delta\in (0,1)$, and define for any $d$-dimensional random vector $Z$,
\[
\| Z\|_M^\star =  \sqrt{\Expect[\| Z_M\|^2 ] } =\sqrt{ \Expect[Z^\transpose M Z] } \, .  
\]
Applying the triangle inequality to \eqref{e:xi_recursion_An} gives
	\begin{equation}
	\begin{aligned}
	\| \xi_{n+1} \|_M 
	& \leq
	(1+2n^{-1} ) \cdot  \| [I+A]    \xi_n  \|_M  +  \| b_{m, n+1}^\xi \|_M +  \| U_{m, n}^{\tilA} \|_M
	\\
	& \leq
	\delta (1+2n^{-1} ) \cdot \|   \xi_n  \|_M  +  \| b_{m,n+1}^\xi \|_M +  \| U_{m, n}^{\tilA} \|_M
	\label{e:xi_recursion_norm_An}
	\end{aligned}
	\end{equation}

Part (ii) of \Lemma{t:bartheta_coupling} implies the following $L_2$ bound:  for some $\overline{\sigma}_\xi < \infty$,
\begin{equation}
	\| b_{m, n+1}^{\xi} \|_M \leq \overline{\sigma}_\xi  \,,  \quad n>m\ge 2\, .
	\label{e:prop_2_bmn}
\end{equation}
	Furthermore, for all $m \geq 2$ and $n > m$, we have:
	\begin{equation}
	\begin{aligned}
	\| U_{m, n}^{\tilA} \|^{\star2}_M 
	& = \Expect\Bigl[ \| \sum_{k=m}^n \frac{1}{k} \tilA_{k+1}  \xi_k \|_M^2  \Bigr]
	\\
	& = \Expect \Big[  \big(\sum_{k=m}^n \frac{1}{k} \tilA_{k+1} \xi_k \big)^\transpose  M \big(\sum_{k=m}^n \frac{1}{k} \tilA_{k+1} \xi_k \big)  \Big ]
	\\
	& = \sum_{k=m}^n \frac{1}{k^2} \Expect \Big[ \xi_k^\transpose  \tilA_{k+1}^\transpose M  \tilA_{k+1} \xi_k \Big ]
	\\
	& \leq \sum_{k=m}^n \frac{1}{k^2} \Expect \Big[ \xi_k^\transpose  \tilA_{k+1}^\transpose M  \tilA_{k+1} \xi_k \Big ]
	\\
	& \leq \overline{\sigma}^2_A \sum_{k=m}^n \frac{1}{k^2}  \| \xi_k \|_M^{\star 2}
	\label{e:prop_2_Umn}
	\end{aligned}
	\end{equation}
where $\overline{\sigma}^2_A$ exists under the boundedness assumption on $\{A_n\}$.
	\ad{It would be nice to actually use these constants in our assumptions!}
	
	Using \eqref{e:prop_2_bmn} and \eqref{e:prop_2_Umn} in \eqref{e:xi_recursion_norm_An} gives:
	\begin{equation}
	\begin{aligned}
	\| \xi_{n+1} \|_M^\star
	& \leq
	\delta (1+2n^{-1} ) \cdot \|   \xi_n  \|_M^\star  +   \overline{\sigma}_\xi +\overline{\sigma}_A \Big(  \sum_{k=m}^n \frac{1}{k^2}  \| \xi_k \|_M^{\star2}\Big)^{\half}
	\\
	& \leq
	\delta (1+2n^{-1} ) \cdot \|   \bfmxi  \|_{M, n}   +   \overline{\sigma}_\xi + \overline{\sigma}_A   \Big( \sum_{k=m}^\infty \frac{1}{k^2} \Big)^{\half} \|   \bfmxi  \|_{M, n}
	\label{e:xi_recursion_norm_An_bound}
	\end{aligned}
	\end{equation}
	where for each $n$,
	\[
	\|  \bfmxi  \|_{M, n} \eqdef  \max_{1 \leq k \leq n} \| \xi_k \|_M  <\infty
	\]
	
Choose  $N_0\ge 2$ such that 
	\[
	\rho = ( 1+2N_0^{-1} ) \delta + \overline{\sigma}_A   \Big( \sum_{k=N_0}^\infty \frac{1}{k^2} \Big)^{\half} < 1,
	\]
and fix $m = N_0$.  
We then  conclude from 	\eqref{e:xi_recursion_norm_An_bound}
 that for each $n \geq N_0$,
	\begin{equation*}
	\begin{aligned}
	\| \xi_{n+1} \| 
	& \leq
	\rho\cdot \|  \bfmxi  \|_{M, n}   +  \overline{\sigma}_\xi
	\label{e:xi_recursion_norm_2_An}
	\end{aligned}
	\end{equation*}
Next apply the definition $\|
 \bfmxi  \|_{M, n+1} = \max \{ \| \xi_{n+1} \| \,,  \|  \bfmxi  \|_{M, n}  \} $,  to obtain
 \begin{equation}
	\|  \bfmxi  \|_{M, n+1} \leq  \max \{ \rho\cdot \|  \bfmxi  \|_{M, n}   + \overline{\sigma}_\xi  \,,  \|  \bfmxi  \|_{M, m}  \} 
	\,,  	\quad n\ge m= N_0
\label{e:xi_recursion_norm_3_An}
	\end{equation}
It then follows by induction that 
\[
	\|  \bfmxi  \|_{M, n} \leq  \max \{     \overline{\sigma}_\xi /(1-\rho)  , \|  \bfmxi  \|_{M, m}   \}  \,, \quad  n\ge N_0\, , 
\]
	which implies \eqref{e:couple}. \qed
%

\section{Variance analysis of the NeSA algorithm}
\label{s:VarNeSA}

Throughout this section it is assumed that the assumptions of \Section{t:OptVarZapHBAn} hold (repeated here for convenience of the reader):

\begin{romannum}
\item[\textbf{(N1)}]  $\{\tilA_n, \tilb_n\}$ are bounded martingale difference sequences.  Moreover, for any 
matrix $Q$,
\[
\Expect [(I+A_n) Q (I+A_n)^{\transpose}\mid \clF_{n-1}] = \clL(Q) 
\]
where $\clF$ is the natural filtration:

\centerline{
$
\clF_n = \sigma\{A_k,b_k : k\le n\}.  
$} 

\item[\textbf{(N2)}]     The bounds in \eqref{e:Aconds} hold, and  the linear operator $\clL$ has  spectral radius strictly bounded by unity.

\item[\textbf{(N3)}] The covariance sequence $\{\Sigma_n\}$ defined in \eqref{e:SigmanAnBlockForm}  is bounded.

\end{romannum}

\spm{Let's drop this for now, but keep discussion that is there. And "initially impose assumption" is strange, given it is assumed in proposition
\\
We will initially impose assumption~(N3) to simplify the proof of convergence. 
Following this proof, we explain how the same arguments leading to convergence can be elaborated to establish boundedness.}
\ad{\rd{Need to direct them to specific subsection where these arguments are made}}

Even under the stability assumption, the convergence proof appears complex.   We first provide a proof for the simpler PolSA algorithm \eqref{e:PolSA_simple} for which (N3) holds by applying \Prop{t:snr}.

%
%
%
%


\subsection{Variance analysis of PolSA}
\label{s:VarAnaZapHB}

The recursion \eqref{e:PolSA_simple} is expressed in state space form as follows:
\begin{gather}
\begin{bmatrix} 
\tiltheta_{n+1} \\ \Delta \theta_{n+1}
\end{bmatrix}
=
\begin{bmatrix}
I & (I + A) \\
0 & (I + A) 
\end{bmatrix}
\begin{bmatrix} 
\tiltheta_n \\ \Delta \theta_n
\end{bmatrix}
+
\alpha_{n+1}
\left[
\begin{bmatrix}
A & 0 \\
A & 0
\end{bmatrix}
\begin{bmatrix}
\tiltheta_n \\
\Delta \theta_n
\end{bmatrix}
+
\begin{bmatrix}
\Noise_{n+1} \\
\Noise_{n+1}
\end{bmatrix}
\right]
\label{e:TilThetaDeltaTheta}
\end{gather}

Recall that the $2d$-dimensional vector process $\{\Phi_n \}$ is defined as:
\begin{equation}
\Phi_n \eqdef  \begin{pmatrix}
\sqrt{n} \tiltheta_n 
\\  
n \Delta \theta_n
\end{pmatrix}
\end{equation}
and the covariance matrix sequence $\{\Sigma_n\}$ is defined to be:
\begin{equation}
\Sigma_n \eqdef \Expect  [  \Phi_n \Phi_n^\transpose]
=
\begin{bmatrix}
\Sigma_n^{11} & \Sigma_n^{12}
\\
\Sigma_n^{21} & \Sigma_n^{22}
\end{bmatrix}
=
\begin{bmatrix}
n \Expect[\tiltheta_n \tiltheta_n^\transpose] & n^{3/2} \Expect[\tiltheta_n \Delta \theta_n^\transpose]
\\
n^{3/2} \Expect[\Delta \theta_n \tiltheta_n^\transpose ] & n^2 \Sigma_n^{22}
\end{bmatrix}
\end{equation}
Also define the noise sequence $\{\Delta_n^{\phi} \}$:
\begin{equation}
\Noise^\Phi_n \eqdef  \begin{pmatrix}
\Noise_n
\\  
\sqrt{n}  \Noise_n
\end{pmatrix} 
\end{equation}

We begin by establishing Assumption~(N3) for PolSA. The following is a direct corollary to \Prop{t:snr} and \Prop{t:couple}:

\begin{proposition}
\label{t:polsa_bdd_var}
Suppose that the assumptions of \Prop{t:couple} hold.   Then, each $\{\Sigma_n^{11}\}$, $\{\Sigma_n^{12}\}$, $\{\Sigma_n^{21}\}$ and $\{\Sigma_n^{22}\}$ are bounded sequences for the PolSA algorithm:  
\[
\sup_n \trace(\Sigma^{ij}_n )  <\infty , \qquad \text{for all }\,\, 1 \leq i,j \leq 2
\]
\end{proposition}
\begin{proof}
\Prop{t:couple} implies that  $\{
n^2 \Expect[\| \tiltheta_n^* - \tiltheta_n \|^2 ] \}$ is bounded in $n$,  
where $\tiltheta_n^*$ is the error sequence of the SNR algorithm, and $\tiltheta_n$ is the error sequence of the PolSA algorithm.
We then have:
\[
\begin{aligned}
\trace(\Sigma_n^{11}) & =   n \| \tiltheta_n\|^2 
\\
&=  n \| \tiltheta_n - \tiltheta_n^*  + \tiltheta_n^* \|^2 
\\
&\leq  2 n \| \tiltheta_n - \tiltheta_n^* \|^2  + 2 n \| \tiltheta_n^* \|^2 
\end{aligned}
\]
Since each of the two terms on the right hand side is bounded (by \Prop{t:couple} and \Prop{t:snr}), we have boundedness of Trace($\Sigma_n^{11}$) for all $n \geq 0$.

Next consider $\{\Sigma^{22}_n\}$. We have:
\[
\begin{aligned}
\trace (\Sigma^{22}_n) 
&=  n^2 \| \Delta \theta_n \|^2
\\
&\leq  2 n^2 \| \Delta \theta_n - \Delta \theta_n^* \|^2 + 2 n^2 \| \Delta \theta_n^* \|^2
\\
&\leq  4 n^2 \| \tiltheta_n - \tiltheta_n^* \|^2  + 4 n^2 \| \tiltheta_{n - 1} - \tiltheta_{n - 1}^* \|^2 + 2 n^2 \| \Delta \theta_n^* \|^2
\end{aligned}
\]
Once again, each of the three terms on the right hand side of the above inequality are bounded, uniformly in $n$, by \Prop{t:couple} and \Prop{t:snr}. The boundedness of ($\Sigma^{12}_n$) and ($\Sigma^{21}_n$) follow from the Cauchy-Schwarz inequality.
\end{proof}

We now proceed to establish the limit for $\{\Sigma_n^{11}\}$.
Using the Taylor series approximation:
\begin{equation}
\sqrt{n+1} = \sqrt{n} + \frac{\sqrt{n}}{2 n} + O\Big(\frac{1}{n^{3/2}}\Big),
\label{e:sqrtn}
\end{equation}
it follows from \eqref{e:TilThetaDeltaTheta} that
the process $\{\Phi_n \}$ evolves as 
\begin{equation}
\Phi_{n+1} = M \Phi_n + \alpha_{n+1} \bigl( B \Phi_n +  {(\sqrt{n+1})}\Noise_{n+1}^{\Phi} + \epsy^\Phi_n \bigr),
\label{e:Phin}
\end{equation}
where the $2d \times 2d$ matrices $M$ and $B$, and the $2d \times 1$ column vector $\epsy^\Phi_n$ are defined to be:
\begin{equation}
M 
\eqdef \begin{bmatrix}
I & \frac{(I + A)}{\sqrt{n}} \\ \\
\frac{ A}{ \sqrt{n}} &  I + A
\end{bmatrix}
\,,
\quad   
B 
\eqdef
\begin{bmatrix}
\half I + A & 0 \\ \\
0 & I + A
\end{bmatrix}
\,,
\quad   
\epsy^\Phi_n
\eqdef
\frac{1}{\sqrt{n}} 
\begin{bmatrix}
{ O(  \| \tiltheta_n \| + \|     \theta_n\|) }
\\
\\
0
\end{bmatrix}
\label{e:MBDef}
\end{equation} 
where the term $\epsy_n^\phi$ is due to the last term in \eqref{e:sqrtn}.

The main step in the proof of \Prop{t:OptVarZapHB} is
to obtain sharp results for  the off-diagonal blocks of the covariance matrix: $\{\Sigma^{12}_n\}$ and $\{\Sigma^{21}_n\}$.   The proof of the following lemma is  contained in Section~\ref{s:recursions}.

\begin{lemma}
\label{t:OffDiagSigma}
Under the conditions of \Prop{t:OptVarZapHB}, for each $n \geq 1$, the following approximations hold for $\Sigma_n$ and  the scaled covariance  $\psi_n \eqdef \sqrt{n}  \Sigma_n^{21}$:
\begin{eqnarray}
\Sigma_{n+1}^{22}   & = &   (I+A)  \Sigma_n^{22}  (I+A)^\transpose + \Sigma^{\Noise}
+ O(\alpha_{n+1}) \,,
\label{e:OffDiagSigma_and_Sigma11_sigma22}
\\
\psi_n   & = & -\Sigma_n^{11}  - A^{-1}  \Sigma_\infty^{22}    + o(1)\,,
\label{e:OffDiagSigma_and_Sigma11_psi}
\\
\Sigma_{n+1}^{11}  &=& \Sigma_n^{11}  + \alpha_{n+1} \Big(  \Sigma_n^{11}   + A \Sigma_n^{11} 
+ \Sigma_n^{11}A^\transpose + \psi_n ^\transpose  (I +A)^\transpose + {(I+A) \psi_n }  
\label{e:OffDiagSigma_and_Sigma11_sigma11}
\\
\nonumber & & \qquad\qquad\qquad\qquad\qquad
+  (I+A) \Sigma_\infty^{22}  (I+A)^\transpose + \Sigma^{\Noise}  + o(1)\Big)
\end{eqnarray}

\end{lemma}

\subparagraph{Proof of \Prop{t:OptVarZapHB}}

From Assumption~(A3), which requires that the eigenvalues of matrix $(I+A)$ lie within the open unit disc, it follows that \eqref{e:OffDiagSigma_and_Sigma11_sigma22} can be approximated by a geometrically stable discrete-time Lyapunov recursion, with a time-invariant, bounded input $\Sigma^{\Noise}$ \cite{kai80lin}.
The limit \eqref{e:SigLim22} directly follows:
\[
\Sigma_\infty^{22} = (I+A) \Sigma_\infty^{22} (I+A)^\transpose + \Sigma^{\Noise}
\]
 
Substituting the approximation \eqref{e:OffDiagSigma_and_Sigma11_psi}
 for $\psi_n$  into the right hand side of the   recursion in \eqref{e:OffDiagSigma_and_Sigma11_sigma11}  gives  (after simplification), 
\begin{equation}
\Sigma_{n+1}^{11}  = \Sigma_n^{11}  + \alpha_{n+1} \Big( - \Sigma_n^{11} + A^{-1} \Sigma^{\Noise} (A^{-1})^\transpose  + o(1)\Big)
\label{e:Sigma11SArecursion}
\end{equation}
This can be regarded as a Euler approximation to the ODE (\cite{bor00a}):
\[
\ddt x_t = -x_t + A^{-1} \Sigma^{\Noise} (A^{-1})^\transpose
\]
The   limits of $\{\Sigma_n^{11} \}$ and $\{ x_t\}$ coincide with the stationary point  $x^* =  A^{-1} \Sigma^{\Noise} (A^{-1})^\transpose$.
\qed

\subsection{Recursion approximations for PolSA}
\label{s:recursions}

The proof of Lemmas~\ref{t:OffDiagSigmaAn}, \ref{t:SAforCovarAn} and \ref{t:OffDiagSigma} are provided here.  These results are established using the result of \Prop{t:polsa_bdd_var}: the covariance sequence $\{\Sigma_n\}$ is bounded for the PolSA algorithm. 


Recall the definitions of $\Sigma_n^\Delta$ and $\Sigma^\Delta$ in \eqref{e:sigma_delta} and \eqref{e:sigma_delta_n}. We begin with the following consequence of the definitions:
\begin{lemma}
\label{t:SigmaNoise}
Under Assumption (N1), the covariance $\Sigma^\Noise_n$ satisfies:
\begin{equation}
\Sigma^\Noise_n = \Sigma^\Noise + O \big(n^{-1/2} \sqrt{\trace({\Sigma^{11}_n})} \big) 
\label{e:SigmaNoiseApprox_old}
\end{equation}
\qed
\end{lemma} 
From \Prop{t:polsa_bdd_var}, we have:
\begin{equation}
\Sigma^\Noise_n = \Sigma^\Noise + O (n^{- 1/2})
\label{e:SigmaNoiseApprox}
\end{equation}

\paragraph{Bounds for PolSA}
\ad{\rd{We are ignoring $\epsy^\Phi_n$ here!! :(}}
From equations \eqref{e:Phin} and \eqref{e:SigmanAnBlockForm}, {ignoring terms that are of the order $O(1/n^{3/2})$}, the $2d \times 2d$ matrix sequence $\{\Sigma_n\}$ satisfies:
\begin{equation}
\Sigma_{n+1} = M \Sigma_n M^\transpose + \alpha_{n+1} (B \Sigma_n  M^\transpose + M \Sigma_n B^\transpose + {\Sigma^{\Noise_\Phi}_{n+1}}),
\label{e:BigCovRec}
\end{equation}
where $M$ and $B$ are defined in \eqref{e:MBDef}, and ${\Sigma^{\Noise_\Phi}_n} = \Expect [\Delta_n^\Phi (\Delta_n^\Phi)^\transpose ]$ is also a $2d \times 2d$ matrix:
\begin{equation}
{\Sigma^{\Noise_\Phi}_n}
\eqdef
\begin{bmatrix}
\Sigma^{\Noise}_n & \sqrt{n} \Sigma^{\Noise}_n \\ \\
\sqrt{n} \Sigma^{\Noise}_n & n \Sigma^{\Noise}_n
\end{bmatrix}
\label{e:SigmaNoiseDef}
\end{equation}
with $\Sigma^\Noise_n$ defined in \eqref{e:sigma_delta_n}. 
Based on \eqref{e:SigmanAnBlockForm}, it is simpler to view \eqref{e:BigCovRec} as four  
parallel {interdependent} matrix recursions:
\ad{\bl{You're right: replaced $o(1)$ with $O(1)$ in the $\Sigma_n^{22}$ recursion. Does this change things?!
\\
And for $\Sigma_n^{12}$ is it $o(1)$; It changes when we use multiply it by $\sqrt{n}$!}
\\
We need to sort this out.}
\begin{equation}
\begin{aligned}
\Sigma_{n+1}^{11}   = & \Sigma_n^{11}  + \frac{1}{\sqrt{n}}  \Big( {\Sigma_n^{12} (I+A)^\transpose} + {(I+A) \Sigma_n^{21} } \Big)  
\\
& 
+ \alpha_{n+1} \Big( (\half I + A) \Sigma_n^{11} + \Sigma_n^{11} (\half I + A)^\transpose + (I+A) \Sigma_n^{22} (I+A)^\transpose +  \Sigma^{\Noise} + \epsy_n^{11} \Big)
\\
\vspace{0.5in}
\\
\Sigma_{n+1}^{12}  = &\Sigma_n^{12} (I+A)^\transpose 
\\
&  + \frac{1}{n} (I+A) \Sigma_n^{21} A^\transpose 
\\
&  + \frac{1}{\sqrt{n} } \Bigl( \Sigma_n^{11} A^\transpose  +  (I+A) \Sigma_n^{22} (I+A)^\transpose + \Sigma^{\Noise}\Bigr) 
\\
&
+ \alpha_{n+1} \Bigg( \frac{(\half I + A) \Sigma_n^{11} A^\transpose}{\sqrt{n}} 
+ \Big( \frac{3}{2} I + A \Big)\Sigma_n^{12} (I + A)^\transpose
+ \frac{(I+A) \Sigma_n^{22} (I + A)^\transpose}{\sqrt{n}}  +  \epsy_n^{12} \Bigg)
\\
\vspace{0.5in}
\\
\Sigma_{n+1}^{22}   =  & (I+A) \Sigma_n^{22} (I+A)^\transpose +  \Sigma^\Noise
+
\frac{1}{n } A \Sigma_n^{11} A^\transpose   +  \frac{1}{\sqrt{n} }  \Bigl( A \Sigma_n^{12} (I+A)^\transpose + (I+A) \Sigma_n^{21} A^\transpose \Bigr)
\\
&
+ \alpha_{n+1} \Bigg( \frac{A \Sigma_n^{12} (I+A)^\transpose}{\sqrt{n}} 
+
\frac{(I+A)\Sigma_n^{21} A^\transpose}{\sqrt{n}}
+ 2 {(I+A) \Sigma_n^{22} (I + A)^\transpose}  +  \epsy_n^{22} \Bigg)
\end{aligned}
\label{e:MatrixRecursions}
\end{equation}
in which the error terms satisfy the following: 
\begin{equation}
\epsy_n^{11} = O(n^{-1/2} \sqrt{\trace{\Sigma^{11}_n}}) \,,\quad
\epsy_n^{12} = O(n^{-1/2} \trace{\Sigma^{11}_n}) \,,\quad
\epsy_n^{22} = O(\trace{\Sigma_n^{11}})
\label{e:epsyn_trace}
\end{equation}
Once again, applying \Prop{t:polsa_bdd_var}, we have:
\begin{equation}
\epsy_n^{11} = O(n^{-1/2}) \,,\quad
\epsy_n^{12} = O(n^{-1/2}) \,,\quad
\epsy_n^{22} = O(1)
\label{e:epsyn}
\end{equation}
\spm{AD writes:  \bl{   Should cite the boundedness theorem here}
\\
SPM:  Let's discuss}
\ad{\rd{NEED TO DEFINE $||| \cdot |||$? AND REPLACE TRACE??}}

\subparagraph{Proof of \Lemma{t:OffDiagSigma}}
The first approximation \eqref{e:OffDiagSigma_and_Sigma11_sigma22}   follows from the recursion for $\Sigma^{22}_n$ in \eqref{e:MatrixRecursions}, and the fact that $\epsy_n^{22} = O(1)$ using \eqref{e:epsyn}.  The stability condition  \eqref{e:Aconds} implies that this sequence is convergent, and the limit solves the Lyapunov equation  
\[
\Sigma_\infty^{22} = (I+A) \Sigma_\infty^{22} (I+A)^\transpose + \Sigma^{\Noise}
\]


We next prove that \eqref{e:OffDiagSigma_and_Sigma11_psi} holds.
Multiplying both sides of the recursion for $\Sigma_n^{21}$ in \eqref{e:MatrixRecursions} by $\sqrt{n+1}$, and using the Taylor series approximation \eqref{e:sqrtn}, we obtain:
\begin{equation}
\begin{aligned}
\psi_{n+1} &  = A \Sigma^{11}_n + (I+A) \psi_n + {(I+A) \Sigma^{22}_n (I+A)^\transpose } + \Sigma^{\Noise} + O(\alpha_{n+1}(1+\|\psi_n\|))
\end{aligned}
\label{e:psin}
\end{equation}
where we have used \Prop{t:polsa_bdd_var} which establishes the boudedness of $\{\Sigma_n\}$.
Notice that the last term in \eqref{e:psin} can be written as:
\[
\begin{aligned}
O(\alpha_{n+1}(1+\|\psi_n\|)) &=  O(\alpha_{n+1}(1+\sqrt{n}  \| \Sigma^{21}_n \| ))
\\
&
=  O(n^{-1/2})
\end{aligned}
\]
where we have once again applied \Prop{t:polsa_bdd_var}, and used $\alpha_{n+1} = 1/(n+1)$.

Substituting the approximation for $\Sigma^{22}_n$ in \eqref{e:OffDiagSigma_and_Sigma11_sigma22} into the recursion \eqref{e:psin} gives
\[
\begin{aligned}
\psi_{n+1}  & =  (I+A) \psi_n + {(I+A) \Sigma^{22}_\infty (I+A)^\transpose } + \Sigma^{\Noise} + O(n^{-1/2})
\\
& = A \Sigma^{11}_n + (I+A) \psi_n +  \Sigma^{22}_\infty + O(n^{-1/2})
\end{aligned}
\] 
Using the principle of super-position for linear systems, we can represent the sequence $\{\psi_n\}$ as the sum of two terms,  with one equal to the ``$O(n^{-1/2})$'' error sequence, and the other evolving as follows:   
\begin{equation}
\begin{aligned}
\psi^\circ_{n+1} & = (I+A) \psi^\circ_n  + u(n)  
\\
u(n) & = A \Sigma^{11}_n + \Sigma^{22}_\infty
\end{aligned}
\label{e:psi3un}
\end{equation}
or equivalently,  
\[
\psi^\circ_{n+k} = (I+A)^k \psi^\circ_n + \sum_{j=0}^{k-1} (I+A)^j u(n+k-1-j)  
\]
The next step is to replace $u(n+k-1-j)$ with $u(n+k)$, and bound the error $\{\epsy(n,k)\}$: 
\begin{equation}
\psi^\circ_{n+k} = (I+A)^k \psi^\circ_n + \sum_{j=0}^{k-1} (I+A)^j u(n+k) + \epsy(n,k)  \,,
\label{e:psi3error}
\end{equation}
where:
\[
\epsy(n,k) = \sum_{j=0}^{k-1} (I+A)^j \big ( u(n+k-1-j) - u(n+k) \big)
\]

From the recursion for $\Sigma^{11}_n$ in \eqref{e:MatrixRecursions},
it follows that for some constant $c_0 < \infty$:
\[
\| \Sigma^{11}_{n+j} - \Sigma^{11}_n \| \leq c_0 \ln \bigg(\frac{n+j+1}{n}\bigg)
\]
Using the bound $\log(1+x)\le x$, it
 follows that for $c < \infty$, the input sequence $\{u(n)\}$ satisfies:
\begin{equation}
\| u(n + j)  -  u(n) \| \leq c \bigg(\frac{j+1}{n}\bigg), \qquad j \geq 0, \quad n \geq 1
\label{e:u_n_bd}
\end{equation}
Using this in the expression for $\epsy(n,k)$, we obtain,
\[
\begin{aligned}
\|\epsy(n,k)\| & \leq  \sum_{j=0}^{k-1} \big \| (I+A)^j u(n+k-1-j) - u(n+k) \big \|
\\
& \leq c  \sum_{j=0}^{k-1} \bigg \| (I+A)^j \bigg( \frac{j+2}{n+k-1-j} \bigg) \bigg \|
\\
& \leq \frac{c}{n}  \sum_{j=0}^{k-1}   \|(I+A)^j\| (j+2)
\\
& \leq \frac{c}{n}  \sum_{j=0}^{\infty}   \|(I+A)^j\|(j+2)
\end{aligned}
\] 
This together with the eigenvalue bound  for $(I+A)$ in \eqref{e:Aconds} gives
\[
\lim_{n \to \infty} \sup_{k} \| \epsy(n,k) \| = 0.
\]
Using this in \eqref{e:psi3error}, we have:
\[
\begin{aligned}
\psi^\circ_{n+k} & = (I+A)^k \psi^\circ_n + \sum_{j=0}^{k-1} (I+A)^j u(n+k) 
\\
& = \sum_{j=0}^{\infty} (I+A)^j u(n+k) + o(1) + O(\rho^k),
\end{aligned}
\]
where $0 < \rho < 1$. Applying \eqref{e:u_n_bd} once more gives
\[
\begin{aligned}
\psi_n 
=
\psi^\circ_n +o(1) 
& = -A^{-1} u(n) + o(1)
\end{aligned}
\]
Substituting the definition of $u(n)$ from \eqref{e:psi3un} gives the desired result:
\[
\psi_{n}   = -\Sigma_n^{11}  - A^{-1}  \Sigma_\infty^{22}    + o(1)
\] 

The final approximation for the $\{\Sigma_n^{11}\}$ recursion in
\eqref{e:OffDiagSigma_and_Sigma11_sigma11} follows 
by substituting $\psi_n = \sqrt{n}\Sigma_n^{21}$ and $\psi_n^{\transpose} = \sqrt{n} \Sigma_n^{21}$ in the recursion for $\Sigma_n^{11}$ in \eqref{e:MatrixRecursions} and then using \eqref{e:OffDiagSigma_and_Sigma11_sigma22} and \eqref{e:OffDiagSigma_and_Sigma11_psi}.
\qed

\newpage

\subsection{Bounds for NeSA}

Assumption (N2) ensures that the following representation is well defined:
\[
[I -\clL]^{-1} = \sum_{k=0}^\infty \clL^k
\]
The matrix $H= [I -\clL]^{-1} (Q)$ solves the Lyapunov equation $H = \clL(H) + Q$ for any matrix $Q$.
The solution to \eqref{e:SigLim22An} is thus  
\begin{equation} 
\Sigma_\infty^{22} = [I -\clL]^{-1}( \Sigma^{\Noise})\,. 
\label{e:SigLim22An_inverse}
\end{equation}

Based on \eqref{e:accSAn_Zap} and \eqref{e:fLin}, the pair of sequences $\{\tiltheta_n\}$ and $\{\Delta \theta_n \}$ for the NeSA algorithm satisfy the following recursion:
\begin{gather}
	\begin{bmatrix} 
		\tiltheta_{n+1} \\ \Delta \theta_{n+1}
	\end{bmatrix}
	=
	\begin{bmatrix}
		I & (I + A_{n+1}) \\
		0 &  (I+ A_{n+1})
	\end{bmatrix}
	\begin{bmatrix} 
		\tiltheta_n \\ \Delta \theta_n
	\end{bmatrix}
	+
	\alpha_{n+1}
	\left[
	\begin{bmatrix}
		A & 0 \\
		A & 0
	\end{bmatrix}
	\begin{bmatrix}
		\tiltheta_n \\
		\Delta \theta_n
	\end{bmatrix}
	+
	\begin{bmatrix}
		\Noise_{n+1} \\
		\Noise_{n+1}
	\end{bmatrix}
	\right]
	\label{e:TilThetaDeltaThetaAn}
\end{gather}

Recall the definitions of $\{\Phi_n\} $ and $\{\Delta_n^{\phi} \}$:
\begin{equation}
\Phi_n 
\equiv \begin{bmatrix}
\Phi^1_n
\\  
\Phi^2_n
\end{bmatrix}
\eqdef  \begin{bmatrix}
\sqrt{n} \tiltheta_n 
\\  
n \Delta \theta_n
\end{bmatrix}
\quad 
\qquad
\Noise^\Phi_n \eqdef  \begin{bmatrix}
\Noise_n
\\  
\sqrt{n}  \Noise_n
\end{bmatrix} 
\end{equation}
Multiplying the first $d$ rows on both sides of \eqref{e:TilThetaDeltaThetaAn} by $\sqrt{n+1}$, and the last $d$ rows by $(n+1)$, and using the Taylor series approximation for $\sqrt{n+1}$ in \eqref{e:sqrtn}, we obtain a state space recursion for the normalized error sequence:
\begin{equation}
\Phi_{n+1} = \haM_{n+1} \Phi_n + \alpha_{n+1} (\haB_{n+1} \Phi_n +  {(\sqrt{n+1})} \Noise_{n+1}^{\Phi} + \epsy^\Phi_n),
\label{e:PhinAn}
\end{equation}
where  the $2d \times 2d$ matrices $\haM_{n+1}$ and $\haB_{n+1}$ are defined as:
\begin{equation}
\haM_{n+1} 
\eqdef \begin{bmatrix}
I &  \frac{1}{\sqrt{n}} {(I + A_{n+1})} \\ \\
\frac{1}{ \sqrt{n}} { A}  &  I + A_{n+1}
\end{bmatrix}
\qquad \qquad
\haB_{n+1}  \eqdef
\begin{bmatrix}
\half I + A  &   0 \\ \\
0 &  I + A_{n+1}
\end{bmatrix}
\label{e:MnBnDef}
\end{equation}
and
\begin{equation}
\epsy^\Phi_n
\eqdef
\frac{1}{\sqrt{n}} 
\begin{bmatrix}
{ O(  \| \tiltheta_n \| + \|     \theta_n\|) }
\\
\\
0
\end{bmatrix}
\end{equation}
To begin the covariance analysis, consider first the outer-products   without   expectation:
\begin{equation}
\haSigma_n \eqdef \Phi_n \Phi_n^\transpose \qquad \qquad  {\haSigma^{\Noise}_n}
\eqdef
\Noise_n \Noise_n^\transpose
\label{e:haSigman}
\end{equation}
along with two more $2d \times 2d$ matrix sequences:
\[
{\haSigma^{\Noise_\Phi}_n}
\eqdef 
\Noise^\phi_n \big(\Noise^\phi_n \big)^\transpose = 
\begin{bmatrix}
\haSigma^{\Delta}_n &  \sqrt{n} \haSigma^{\Delta}_{n} \\ \\
\sqrt{n} \haSigma^{\Delta}_n  &  n  \haSigma^{\Delta}_n
\end{bmatrix}
\quad
\textit{and}
\quad
W^\transpose_{n+1} \eqdef  
(\sqrt{n+1}) \Phi_n \big( \Delta_{n+1}^{\phi} \big)^\transpose
\]

Then, based on \eqref{e:PhinAn}, {ignoring terms of the order $O(1/n^{3/2})$}, the $2d \times 2d$ matrix sequence $\{\haSigma_n\}$ satisfies the following recursion:
\begin{equation}
\begin{aligned}
\haSigma_{n+1} = \haM_{n+1} \haSigma_n \haM_{n+1}^\transpose + \alpha_{n+1} \Big( & \haB_{n+1} \haSigma_n  \haM_{n+1}^\transpose + \haM_{n+1} \haSigma_n \haB_{n+1}^\transpose
\\ 
& + {\haSigma^{\Noise_\Phi}_{n+1}} + \haM_{n+1} \haW_{n+1}^\transpose + \haW_{n+1} \haM_{n+1}^{\transpose} \Big).
\end{aligned}
\label{e:BigCovRechaSigma}
\end{equation}

Taking expectations results in these interdependent matrix recursions:
\begin{equation}
\begin{aligned}
\Sigma_{n+1}^{11} &  = \Sigma_n^{11}  
+ \frac{1}{\sqrt{n}} \Bigl( \Sigma_n^{12} (I+A)^\transpose +  (I+A) \Sigma_n^{21} \Bigr)
\\
&
+ \alpha_{n+1} \Big( (\half I + A) \Sigma_n^{11} + \Sigma_n^{11} (\half I + A)^\transpose + \ell^{22}_n + \Sigma^{\Noise} + \Sigma^W_n + {\Sigma^W_n}^\transpose + o(1) \Big)
\\
\vspace{0.5in}
\\
\Sigma_{n+1}^{12} &  = {\Sigma_n^{12} (I+A)^\transpose} +  \frac{1}{n} (I+A) \Sigma_n^{21} A^\transpose
+  \frac{1}{\sqrt{n}} \Bigl( \Sigma_n^{11} A^\transpose +  \ell_n^{22}   +  \Sigma^{\Noise} + \Sigma^W_n + {\Sigma^W_n}^\transpose  \Bigr)
\\
&
+ \alpha_{n+1} \Bigg( \frac{(\half I + A) \Sigma_n^{11} A^\transpose}{\sqrt{n}} 
+ \Big( \frac{3}{2} I + A \Big)\Sigma_n^{12} (I + A)^\transpose
+ \frac{ \ell^{22}_n }{\sqrt{n}}
+ o(1) \Bigg)
\\
\vspace{0.5in}
\\
\Sigma_{n+1}^{21} &  = (\Sigma_{n+1}^{12})^\transpose
\\
\vspace{0.5in}
\\
\Sigma_{n+1}^{22} &  =   \ell^{22}_n 
+ \bigg (\Sigma^{\Noise} + \Sigma^W_n + {\Sigma^W_n}^\transpose \bigg)
+ \frac{A \Sigma_n^{12} (I+A)^\transpose}{\sqrt{n}} + \frac{(I+A) \Sigma_n^{21} A^\transpose}{\sqrt{n}}
+ \frac{A \Sigma_n^{11} A^\transpose}{n}  
\\
&
+ \alpha_{n+1} \Bigg( \frac{A \Sigma_n^{12} (I+A)^\transpose}{\sqrt{n}} 
+
\frac{(I+A)\Sigma_n^{21} A^\transpose}{\sqrt{n}}
+ 2  \ell_n^{22} + O(1) \Bigg)
\end{aligned}
\label{e:MatrixRecursionsAn}
\end{equation}
where $
\Sigma^W_n \eqdef \Expect [ (I + A_{n+1}) \Phi_n^1 \Noise_{n+1}^\transpose  ]
$,  and
$\ell^{22}_n = \Expect [ \clL(\haSigma_n^{22}) ]$ with   $\clL$ is defined in \eqref{e:clL}.

Linearity implies that expectation and operation can be interchanged:
\[
\ell^{22}_n =   \clL(\Expect [ \haSigma_n^{22} ] ) = \clL(\Sigma_n^{22})
\]

Two more linear operators are required in the following:     
For any $Q \in \Re^{d\times d}$ define 
\begin{equation}
\tilclL(Q) \eqdef  \Expect [ (\tilA_{n+1}) {Q} \tilA_{n+1}^\transpose]\,.
\label{e:bfmtilEll}
\end{equation}
The second operator maps vectors to matrices: for any $v \in \Re^d$,
\begin{equation}
\clM(v) = \Expect [ (I + A_{n+1}) v (\tilA_{n+1} \theta^* - \tilb_{n+1} )^\transpose  ]
\label{e:clM}
\end{equation}

The following result will be used to show that $\{\Sigma^W_n\}$ converges to $0$:
\begin{lemma}
	Under (N1)--(N3) we have  $\lim_{n\to\infty}  
	\Expect[\Phi_n] =0$.  
	\label{t:Phin}
\end{lemma}

\begin{proof}
	Denote $\barPhi_n = \Expect[ \Phi_n]$, and similarly, $\barPhi^1_n = \Expect[ \Phi^1_n]$, $\barPhi^2_n = \Expect[ \Phi^2_n]$.
	Based on \eqref{e:PhinAn}, we have:
	\begin{equation}
	\begin{aligned}
	\barPhi^2_{n+1} = & \,\, \Big (1 + \frac{1}{n + 1} \Big ) \big(I + A) \barPhi_n^2 + \frac{1}{\sqrt{n}} A \barPhi_n^1
	\\
	\approxeq  &  \sum_{k = 0}^{n - 1} \big(I + A)^k \frac{1}{\sqrt{n - k}} A \barPhi_{n - k}^1
	\\
	\approxeq  &   \frac{1}{\sqrt{n}} \bigg(  \sum_{k = 0}^{\infty} \big(I + A)^k \bigg) A \barPhi_{n}^1
	\\
	=  &   - \frac{1}{\sqrt{n}} A^{-1} A \barPhi_{n}^1
	\\
	=  &   - \frac{1}{\sqrt{n}} \barPhi_{n}^1
	\end{aligned}
	\label{e:Phin2}
	\end{equation}
	In \eqref{e:Phin2}, each of the approximations can be shown rigorously, using techniques that are very similar to the ones that were used in \Section{s:recursions}: The error in each approximation can be bounded as a constant times $\|\Phi_n\|/n$. 
	
	Similarly, the recursion for $\barPhi^1_n$ satisfies:
	\begin{equation}
	\begin{aligned}
	\barPhi^1_{n+1} = & \,\, \barPhi_n^1 + \frac{1}{\sqrt{n}} (I + A) \barPhi_n^2 +  \frac{1}{n + 1} \big( \half I + A) \barPhi_n^1
	\\
	\approxeq  &  \barPhi_n^1 - \frac{1}{{n}} (I + A) \barPhi_n^1 +  \frac{1}{n + 1} \big( \half I + A) \barPhi_n^1
	\\
	\approxeq  &  \barPhi_n^1 -  \frac{1}{2(n + 1)} \barPhi_n^1
	\end{aligned}
	\end{equation}
	The above recursion can be viewed as a stochastic approximation algorithm (with $0$ noise). This implies:
	$
	\lim_{n\to\infty}  
	\Expect[\Phi^1_n] =0
	$.  Using this in \eqref{e:Phin2}, we also have $
	\lim_{n\to\infty}  
	\Expect[\Phi_n] =0
	$.
\end{proof}

\begin{lemma}
	Under (N1)--(N3) we have  $\lim_{n\to\infty}  
	\Sigma^W_n =0$.  
	\label{t:SigmaWn}
\end{lemma}

\begin{proof}  
	Using the definition of $\Noise_n$ in \eqref{e:SigmaNoise} gives
	\[
	\begin{aligned}
	\Sigma^W_n & 
	= \Expect [ (I + A_{n+1}) \Phi_n^1 \tiltheta_n^\transpose \tilA_{n+1}^\transpose] + \Expect [ (I + A_{n+1}) \Phi_n^1 (\tilA_{n+1} \theta^* - \tilb_{n+1})^\transpose  ]
	\\
	\\
	& = \Expect [ (\tilA_{n+1}) \Phi_n^1 \tiltheta_n^\transpose \tilA_{n+1}^\transpose] + \Expect [ (I + A_{n+1}) \Phi_n^1 (\tilA_{n+1} \theta^* - \tilb_{n+1} )^\transpose  ]
	\\
	\\
	& = \frac{1}{\sqrt{n}} \Expect \Big[ \Expect \big [ (\tilA_{n+1}) {\haSigma_n^{21}} \tilA_{n+1}^\transpose \big | \clF_n \big]  \Big] + \Expect \Big [ \Expect \big [(I + A_{n+1}) \Phi_n^1 (\tilA_{n+1} \theta^* - \tilb_{n+1} )^\transpose \big | \clF_n \big ]  \Big ]
	\\
	\\
	& = \frac{1}{\sqrt{n}} \Expect \big [ \tilclL({\haSigma_n^{21}})  \big ] + \Expect \big [ \clM(\Phi_n^1) \big ]
	\end{aligned}
	\]
	where the first equality follows from the fact that $\Expect [ \Phi_n^1 \tiltheta_n^\transpose \tilA_{n+1}^\transpose] = 0$,  second equality follows from the definition of $\haSigma^{21}$, and the last equality follows from (N1) and definitions \eqref{e:bfmtilEll} and \eqref{e:clM}.

	Linearity of $\tilclL$ and $\clM$ then implies
	\[
	\begin{aligned}
	\Sigma^W_n
	& = \frac{1}{\sqrt{n}}  \tilclL( \Expect[{\haSigma_n^{21}}]) +   \clM( \Expect [\Phi_n^1] )
	\\
	& = \frac{1}{\sqrt{n}}  \tilclL({\Sigma_n^{21}}) +   \clM( \Expect [\Phi_n^1] )
	\\
	&= \frac{1}{n} \tilclL({\psi_n}) + \clM(\Expect[\Phi_n^1])
	\end{aligned}
	\]
	where the last equality used the definition  $\psi_n = \sqrt{n} \Sigma^{21}_n$.
	
	The first term vanishes under (N3) since $\psi_n/n=O(n^{-1/2})$.


	It remains  to show that the second term converges to $0$.
	Based on \eqref{e:PhinAn} it is straightforward to establish the following limit under (N1)--(N2):
	\[
	\lim_{n\to\infty}\Expect[\Phi_n] =0
	\]

	\ad{\rd{I rechecked; It looks like the $(1/\sqrt{n})$ comes from $\clM(\Delta_{n+1})$. Should I write it out? Or just mention it? I will mention it for sure.  
			\\
			Things are a bit subtle! Looks like we have to say why $\Expect[\clM(\Phi_n^1)] = \clM(\Expect[\Phi_n^1])$ but not the case for $\Delta_{n+1}$.
			\\
			SPM does not understand}}
	Using the definition \eqref{e:clM}, and taking expectations completes the proof that $\Expect  [ \clM(\Phi_n^1)   ]  = o(1)$. 
\end{proof}

\spm{\rd{Urgent!}  How did you eliminate $Z_n=\Phi^2_n$?
	To resolve this, 
	I have cheated and simply stated it is obvious -- surely is!  But I haven't checked.}
\spm{$\Delta_n$ has mean zero -- martingale difference.
	\\
	\rd{where the $O(1/\sqrt{n})$ terms is due to $\Expect[\clM(\Delta_{n+1})]$.}
}

\vspace{-0.1in}

\subparagraph{Proof of \Lemma{t:OffDiagSigmaAn}}
We first prove the first recursion in \eqref{e:OffDiagSigmaAn}.
Based on the assumption that each $\Sigma^{11}$, $\Sigma^{22}$ and $\psi_n = \sqrt{n} \Sigma_n^{21}$ in \eqref{e:MatrixRecursionsAn} are bounded, the recursion for $\Sigma_n^{22}$ can be written as
\begin{equation*}
	\begin{aligned}
		\Sigma_{n+1}^{22}  =   \clL \big( \Sigma_n^{22} \big) + \Sigma^{\Noise} + \Sigma^W_n + { \Sigma^W_n }^\transpose
		+ O(\alpha_{n+1}),
	\end{aligned}
	\label{e:Sigma22nAn}
\end{equation*}
where the $O(\alpha_{n+1})$ term includes all terms in the recursion that are multiplied with $\alpha_{n+1}$.
\Lemma{t:SigmaWn} then implies the first recursion in \eqref{e:OffDiagSigmaAn},
\[
\Sigma_{n+1}^{22}  =   \clL \big( \Sigma_n^{22} \big) + \Sigma^{\Noise} + O(1/\sqrt{n})
\]

\spm{please no contraction operators!
	\bl{which states that $\clL(\cdot)$ is a contraction operator},
}

This is regarded as a Lyapunov recursion with time varying forcing term $\Sigma^{\Noise} + O(1/\sqrt{n})$.
Under (N2) convergence follows, giving \eqref{e:SigLim22An}.

We next prove that the second recursion in \eqref{e:OffDiagSigmaAn} holds.
Multiplying both sides of the recursion for $\Sigma_n^{21}$ in \eqref{e:MatrixRecursionsAn} by $\sqrt{n+1}$, and using the Taylor series approximation \eqref{e:sqrtn}, we obtain:
\[
\begin{aligned}
\psi_{n+1} &  = (I+A) \psi_n + A \Sigma_n^{11} +  \clL(\Sigma_n^{22}) + \Sigma^{\Noise} + \Sigma^W_n + {\Sigma^W_n}^\transpose + O(\alpha_{n+1}),
\end{aligned}
\]
where once again, the $O(\alpha_{n+1})$ terms are due to the boundedness assumption on $\Sigma^{11}$, $\Sigma^{22}$ and $\psi_n$.
From \Lemma{t:SigmaWn}, we have $\Sigma_n^W \to 0$, and furthermore, using $\Sigma_n^{22} = \Sigma_{\infty}^{22} + o(1)$,
\[
\begin{aligned}
\psi_{n+1}  & =  (I+A) \psi_n + A \Sigma_n^{11} + \clL(\Sigma_{\infty}^{22}) + \Sigma^{\Noise} + o(1)
\\
& = (I+A) \psi_n  +  A \Sigma_n^{11} + \Sigma_{\infty}^{22} + o(1)
\end{aligned}
\]
Once we have the above form, the rest of the proof follows steps exactly same as the proof of \Lemma{t:OffDiagSigma}: By viewing the recursion as state evolution of a discrete-time stable linear system with bounded input sequence $\{A \Sigma_n^{11} + \Sigma_{\infty}^{22}\}$ and vanishing additive noise, we can show that $\{\psi_n\}$ satisfies the second recursion in \eqref{e:OffDiagSigmaAn}.
\qed

\subparagraph{Proof of \Lemma{t:SAforCovarAn}}
The Lemma follows directly by substituting $\psi_n = \sqrt{n}\Sigma_n^{21}$ and $\psi_n^{\transpose} = \sqrt{n} \Sigma_n^{21}$ in the recursion for $\Sigma_n^{11}$ in \eqref{e:MatrixRecursionsAn} and then using \eqref{e:OffDiagSigmaAn}.
\qed

\newpage

\section{NeSA and PolSA TD-learning algorithms}
\label{s:TD_appendix}

\def\elig{\zeta}
\def\bfelig{\bfmath{\zeta}}

\def\diagpie{\Pi}
\def\pie{\varpi}

\def\hab{\widehat b}

\def\Ascale{\zeta}

\def\barf{\overline{f}}

In this section of the Appendix, we briefly give details on how to apply the algorithms introduced in this paper to solve value-function estimation problems in Reinforcement Learning. For simplicity, we  consider the TD($\lambda$)-learning algorithm with $\lambda = 0$. Extension to $\lambda \in (0,1]$ is straightforward.

Consider a Markov chain $\bfmX$ evolving on $\state \in \Re^\ell$. 
Let $\{P^n\}$ denote its transition semigroup:
For each $n\ge 0$, $x\in\state$, and  $A\in \clB(\state)$ (where $\clB(\cdot)$ denotes the Borel set),
\[
P^n(x,A):=\Prob_x\{X_n\in A\}:=\Pr\{X_n\in A\,|\,X_0=x\}.
\]
The standard operator-theoretic notation is used for conditional expectation:  
for any measurable function $f \colon \state \to \Re$,  
\[
P^n f\, (x) = 
\Expect_x [f(X_n)  ]  \eqdef
\Expect [f(X_n) \mid X_0 = x].
\]
In a finite state space setting, $P^n$ is the $n$-step transition probability matrix of the Markov chain, and the conditional expectation appears as matrix-vector multiplication:
\[
P^n f\, (x) = \sum_{x'\in\state} P^n(x,x') f(x'),\qquad x\in\state.
\]

Let $c\colon\state\to\Re_+$ denote a cost function, and $\beta\in (0,1)$ a discount factor.
The discounted-cost value function is defined as  
\[
h(x) = \sum_{n=0}^\infty \beta^n P^n c(x)\,, \qquad  x \in \state
\] 
It is known that the value function is the unique solution to the Bellman equation
\begin{equation}
c(x) + \beta P h(x) = h (x)
\label{e:DiscFish}
\end{equation}
Consider the case of a $d$-dimensional linear parameterization: A function $\basis \colon\state\to\Re^d$ is chosen,  which is viewed as a collection of $d$ basis functions. Given a parameter vector $\theta \in \Re^d$, the corresponding approximation to the value function is defined as:
\[
h^\theta (x) = \sum_i \theta_i \basis_i (x) = \theta^\transpose \basis(x)
\]

The goal of TD-learning is to approximate the solution to \eqref{e:DiscFish} by $h^\theta(x)$ \cite{sut88,tsiroy97a}. In particular, the TD($0$) algorithm intends to solve the Galerkin relaxation of the problem \cite{sze10,devmey17a}: Find $\theta^*$ such that
\begin{equation} 
0 = \Expect\bigl[  \bigl( -{{{h}}}^{\theta^*} (X_n) + c(X_n)+ \beta {{{h}}}^{\theta^*} (X_{n+1}) \bigr)   \basis_i(X_n) \bigr] \, , \quad 1\le i\le d\,,
\label{e:MetricGalTD} 
\end{equation}
where the expectation is with respect to the steady state distribution of the Markov chain.
This model is of the form considered in \Prop{t:snr},  with Markov chain defined by 
$\process_n= (X_n, X_{n-1})$, and $\barf(\theta) = A\theta - b$, with  $A\eqdef \Expect[A_n ] $,   $b\eqdef   \Expect [b_n]  $ (expectations in steady-state), and
\begin{equation}
A_{n}  \eqdef \basis(X_{n-1}) \big(  \beta \basis(X_{n}) - \basis(X_{n-1})\big)^\transpose \,,  \qquad   b_n  \eqdef - \basis(X_n) c(X_n)
\label{e:AbTDdef}
\end{equation}

The TD($0$) algorithm is stochastic approximation in the form \eqref{e:SAa}, with $G_n \equiv I$.
\paragraph{TD($0$) algorithm:}
For initialization $\theta_0 \in \Re^d  $, the sequence of estimates are defined recursively:  
\begin{equation}
\begin{aligned}
\theta_{n+1} & = \theta_n + \alpha_{n+1} \basis(X_n)  d_{n+1} 
\\
d_{n+1} & =  c(X_n) +  \bigl[\beta \basis(X_{n+1}) -  \basis(X_n) \bigr]^\transpose  \theta_n
\label{e:TD0}
\end{aligned}
\end{equation}
The sequence $\{d_n\}$ also appears in the algorithms described next.    
The parameter recursion can  be expressed in the more suggestive form
\[
\theta_{n+1}  = \theta_n + \alpha_{n+1} [A_{n+1}\theta_n-b_n]
\]

\ad{\rd{Important! Do I need to talk about invertibility and all those things!?}
	\\
	SM to AD:  minimal description needed}
The LSTD algorithm of \cite{boy02} is   a stochastic approximation algorithm of the form \eqref{e:SAa}, with $G_n$ equal to a Monte-Carlo estimate of   $-A^{-1}$   \cite{devmey17a}.
\paragraph{LSTD($0$) algorithm:}
For initialization $\theta_0 \in \Re^d $, the sequence of estimates are defined recursively:  
\begin{equation}
\begin{aligned}
\theta_{n+1} & = \theta_n - \alpha_{n+1} \haA_{n+1}^{-1} \basis(X_n)  d_{n+1} 
\\
d_{n+1} & =  c(X_n) +  \bigl[\beta \basis(X_{n+1}) -  \basis(X_n) \bigr]^\transpose  \theta_n
\\
\haA_{n+1}  & = \haA_n + \gamma_{n+1} [A_{n+1}-\haA_n]
\label{e:LSTD0}
\end{aligned}
\end{equation}
The non-negative gain sequence $\{\gamma_n\}$ is an ingredient in the ``Zap'' algorithms of \cite{devmey17a,devmey17b},  where it is assumed to satisfy standard assumptions, but is relatively large:
\[
\sum_{n=1}^\infty \gamma_n=\infty \,,
\quad             
\sum_{n=1}^\infty \gamma_n^2 < \infty \,,
\qquad
\lim_{n\to\infty}  \frac{\gamma_n}{\alpha_n} = \infty
\]
There is a single gain sequence in the LSTD($0$) algorithm, $\alpha_n\equiv \gamma_n$,  and in this case the matrix recursion is equivalent to the simple average:
\[
\haA_n = \frac{1}{n} \sum_{i=1}^{n} A_i = \frac{1}{n} \sum_{i=1}^{n} \basis(X_{i-1}) \bigl[\beta \basis(X_{i}) -  \basis(X_{i-1}) \bigr]^\transpose 
\]

The PolSA and NeSA algorithms for TD($0$)-learning are given as follows.  In each case, 
the initialization $\theta_0 \in \Re^d $,  and gain $\Ascale$ satisfying \eqref{e:Aconds}, are pre-specified.
\paragraph{PolSA TD($0$) algorithm:}  
\begin{equation}
\begin{aligned}
\theta_{n+1} & = \theta_n + (I + \Ascale {\haA_{n+1} } ) \Delta \theta_n + \alpha_{n+1} \Ascale \basis(X_n)  d_{n+1} 
\\
d_{n+1} & =  c(X_n) +  \bigl[\beta \basis(X_{n+1}) -  \basis(X_n) \bigr]^\transpose  \theta_n
\label{e:PolSATD0}
\end{aligned}
\end{equation}

\paragraph{NeSA TD($0$) algorithm:} 
\begin{equation}
\begin{aligned}
\theta_{n+1} & = \theta_n + (I + \Ascale {A_{n+1}}) \Delta \theta_n + \alpha_{n+1} \Ascale \basis(X_n)  d_{n+1} 
\\
d_{n+1} & =  c(X_n) +  \bigl[\beta \basis(X_{n+1}) -  \basis(X_n) \bigr]^\transpose  \theta_n
\\
\label{e:NeSATD0}
\end{aligned}
\end{equation}
The matrix momentum term,   highlighted in red,  is the only  difference between the two algorithms.

\end{document}